\documentclass[10pt]{amsart}
\usepackage{amssymb,amsmath,epsfig,graphics,latexsym,psfrag}
\usepackage[english]{babel}
\usepackage[all]{xy}
\usepackage{amscd, amssymb, amsmath, amsthm, graphics}
\usepackage{amsmath,amsfonts,amsthm,amssymb}
\usepackage{latexsym,amsmath}
\usepackage{graphicx,psfrag}
\usepackage{mathrsfs}

\newtheorem{theorem}{Theorem}[section]
\newtheorem{proposition}[theorem]{Proposition}
\newtheorem{corollary}[theorem]{Corollary}
\newtheorem{lemma}[theorem]{Lemma}

\theoremstyle{definition}

\newtheorem{example}[theorem]{Example}

\theoremstyle{remark}
\newtheorem{remark}[theorem]{Remark}

\newtheorem{question}{Question}
\numberwithin{equation}{section}
\renewcommand{\t}{ \widetilde}
\renewcommand{\hat}{ \widehat}
\renewcommand{\b}{ \partial}
\newcommand{\Z}{\mathbb Z}
\newcommand{\R}{\mathbb R}
\newcommand{\N}{\mathbb N}
\newcommand{\C}{\mathbb C}

\newcommand{\Hi}{\bf H}

\renewcommand{\S}{\bf S}
\renewcommand{\l}{\langle}
\renewcommand{\r}{\rangle}
\newcommand{\e}{\varepsilon}
\newcommand{\z}[1]{{\Z}/#1{\Z}}
\renewcommand{\o}{\overline}

\newcommand{\co}{\colon\thinspace}
\renewcommand{\epsilon}{\varepsilon}
\renewcommand{\c}{\mathcal}

\usepackage{times}

\usepackage{times}

\begin{document}
\sloppy

\title[]{Chern Simons Theory and the volume  of $3$-manifolds}

\author{Pierre Derbez}
\address{Universit\'e Aix-Marseille, LATP UMR 6632 de CNRS,
CMI, Technopole de Chateau-Gombert,
39, rue Fr\'ed\'eric Joliot-Curie -
 13453 Marseille Cedex 13}
\email{pderbez@gmail.com}

\author{Shicheng Wang}
\address{Department of Mathematics, Peking University, Beijing, China}
\email{wangsc@math.pku.edu.cn}



\date{\today}
\maketitle

\begin{abstract}
We give some applications of the  Chern Simons gauge theory to the
study of the set ${\rm vol}\left(N,G\right)$  of  volumes  of all
representations $\rho\co\pi_1N\to G$, where $N$ is a closed oriented
three-manifold and $G$ is either ${\rm Iso}_e\t{\rm SL_2(\R)}$, the
isometry group of the Seifert geometry,  or ${\rm Iso}_+{\Hi}^3$,
the orientation preserving isometry group of the hyperbolic 3-space.
We focus on three natural questions arising from the definition of
${\rm vol}\left(N,G\right)$:

(1) How to find non-zero values in ${\rm vol}\left(N, G
 \right)$? or weakly how to find  non-zero elements in ${\rm
vol}(\t N, G)$ for some finite cover $\t N$ of $N$?

(2) Do these volumes    satisfy the  covering
property in the sense of Thurston?

(3) What kind of topological information  is enclosed in the
elements of ${\rm vol}(N, G)$?

By various methods of computations, involving relations between the
volume of representations and the Chern-Simons invariants of flat
connections, we are able to give several meaningful results related
the questions above.

We  determine ${\rm vol}\left(N, G \right)$ when $N$ supports the
Seifert geometry, and we find some non-zero values in ${\rm
vol}\left(N,G\right)$ for  certain 3-manifolds with non-trivial
geometric decomposition for either  $G={\rm Iso}_+{\Hi}^3$ or ${\rm
Iso}_e\t{\rm SL_2(\R)}$. Moreover we will show that unlike the
Gromov simplicial volume,  these non-zero elements carry
 the gluing information
between the geometric pieces of $N$.

For a large class 3-manifolds $N$, including all
rational homology 3-spheres, we prove that $N$ has a  positive Gromov simplicial volume iff it admits 
 a finite covering  $\t N$ with ${\rm vol}(\t
N,{\rm Iso}_+{\Hi}^3)\ne \{0\}$. On the other hand, among such class,
there are some $N$ with positive simplicial volume but ${\rm
vol}(N,{\rm Iso}_+{\Hi}^3)=\{0\}$, 
 yielding a negative answer to question (2) for hyperbolic volume.
\end{abstract}

 \vspace{-.5cm}
\tableofcontents

\section{Introduction}

The volume of representations of 3-manifolds groups is a beautiful
theory which has  rich connections with many branches of
mathematics. However the behavior of those volume functions seem
still quite mysterious. To make our meaning more explicit, we first
give some basic notions (which will be defined later) and properties
of the volume of representations. Let $N$ be  a closed oriented
$3$-manifold. Let $G$ be either ${\rm PSL}(2;{\C})={\rm
Iso}_+{\Hi}^3$, the orientation preserving isometry group of the
hyperbolic 3-space, or ${\rm Iso}_e\t{\rm SL_2(\R)}$, the identity
component of the isometry group of  $\t{\rm SL_2(\R)}$. For each
representation $\rho\co\pi_1N\to G$, the volume of
 $\rho$ is denoted by ${\rm vol}_G(N,\rho)$. We denote by ${\rm
vol}\left(N,G\right)$ the set
$$\{{\rm vol}_G(N,\rho),\ {\rm when}\ \rho\
{\rm runs\ among\  the\ representations}\ \rho\co\pi_1N\to G\}$$
 Suppose
$N$ supports the hyperbolic, resp. $\t{\rm SL_2(\R)}$, geometry. Then
$N$ naturally has its own hyperbolic volume ${\rm vol}_{{\Hi}^3}(N)$,
resp. Seifert volume ${\rm vol}_{\t{\rm SL_2(\R)}}(N)$. We denote by
$||N||$  the Gromov simplicial volume of $N$, which measures, up to a
multiplicative constant, the hyperbolic volume of the hyperbolic
pieces of $N$ (see \cite{G}).
The following theorem contains basic results of the theory of volume
representations. For its development, see \cite{BG1}, \cite{BG2},
\cite{Re2}, \cite{Re1} and their references.

\begin{theorem}\label{basic property of volume of presentation}
Let $N$ be a  closed oriented $3$-manifold.

(1)  Both ${\rm vol}(N, {\rm PSL}(2;{\C}))$ and ${\rm vol}(N, {\rm
Iso}_e\t{\rm SL_2(\R)})$ contain at most finitely many values and we denote by $HV(N)$ and $SV(N)$ the maximum value for ${\rm PSL}(2;{\C})$ and ${\rm
Iso}_e\t{\rm SL_2(\R)}$ respectively.

(2) Suppose  $N$ supports a hyperbolic geometry. Then ${\rm
vol}_{{\Hi}^3}(N)$ is reached by ${\rm vol}_{\rm PSL(2;{\C})}(N,\rho)$
for some discrete and faithful representation. The similar statement
is still true when $N$ supports an $\t{\rm SL_2(\R)}$-geometry.

(3) ${\rm vol}_{{\rm PSL}(2;{\C})}(N,\rho)\leq \mu_3||N||,$ where $\mu_3$ denotes the volume of any ideal regular tetrahedron in ${\Hi}^3$.

(4)  Let $f\co M\to N$ be a map of degree $d$ and let $\rho: \pi_1 N\to G$
denote a representation. Then we get a representation $\rho\circ f_{\ast}\co\pi_1M\to G$ such that ${\rm vol}_G(\rho\circ f_{\ast}, M)= d{\rm
vol}_G(\rho, N)$. Accordingly this yields  the classical inequalities $$HV(M)\geq|{\rm deg}f|HV(N)\   and\  SV(M)\geq|{\rm deg}f|SV(N)$$
\end{theorem}

\begin{remark} Recall that a prime 3-manifold $N$ admits no  self-map of degree $>1$
if and only if either $N$ has a non-trivial geometric decomposition,
or
 supports an $\t{\rm SL_2(\R)}$ or a hyperbolic geometry.
This fact combined with  Theorem \ref{basic property of volume of
presentation} (2), (3), (4)   implies that if  ${\rm vol}\left(N,
\rm {Iso}_e\t{\rm SL_2(\R)}\right)\not=\{0\}$ then necessarily
either the geometric decomposition of $N$ is non trivial, or $N$
supports an $\t{\rm SL_2(\R)}$  or a hyperbolic geometry and if ${\rm
vol}\left( N, \rm PSL(2;{\C}) \right)\not=\{0\}$ then necessarily
$N$ contains some hyperbolic piece(s).
\end{remark}

Besides Theorem \ref{basic property of volume of presentation},
Thurston pointed out  the relation between Chern-Simon's invariants
and the hyperbolic volume  of hyperbolic 3-manifolds for discrete
and faithful representations \cite{Th2}. Later such a relation is
extended in \cite{KK} for cusped hyperbolic 3-manifolds and  discrete
and faithful representations into ${\rm PSL}(2;{\C})$, and by \cite{Kh} for closed manifolds
with the group $\t{\rm SL_2(\R)}$ (as a subgroup of ${\rm Iso}\t{\rm
SL_2(\R)}$).

Despite  those significant results, the questions below, which is
a main motivation of  this paper, seems still remarkably unknown.

\begin{question}\label{1}
 (1) (i) How to find  non-zero elements in ${\rm vol}\left(N, G
\right)$? or weakly,

(ii)  How to find  non-zero elements in ${\rm vol}(\t N, G)$
for some finite cover $\t N$ of $N$?

 (2) Does $HV$ or $SV$ satisfy
the so-called "covering property" in the sense of Thurston?

(3) What kind of topological information of  $N$ is captured by the
non-zero elements in ${\rm vol}\left(N, G \right)$?

\end{question}

We recall that a non-negative 3-manifolds invariant $\eta$  satisfies the {\it covering
property}, if for any finite covering $p: \t N \to N$, we have
$\eta(\t N)= |\text{deg}(p)|  \eta(N)$.

\begin{remark} Three-manifold invariants with the
covering property  was first addressed by Thurston  in \cite[Problem
3.16]{Ki}.  The simplicial volume has the covering property
(Gromov-Thurston-Soma). Some papers sought   invariants with
covering property for graph manifolds, as in \cite{WW}.

So far it seems that we only know that $HV$, resp. $SV$, satisfies
the covering property for the hyperbolic, resp. Seifert, manifolds.
In hyperbolic geometry this property comes from the relation between
the simplicial volume and $HV$. In Seifert geometry one can compute
$SV$ in terms of the Euler
 classes of the Seifert manifold and the Euler characteristic of its orbifold   and these
invariants behave naturally under covering maps.
\end{remark}

From now on the manifolds are assumed  closed oriented and  irreducible.

\subsection{Volumes of Seifert manifolds}

The works of Brooks-Goldman  \cite{BG2}, Milnor, Wood \cite{Mi},
\cite{Wo} and  Eisenbud-Hirsch-Neumann  \cite{EHN}, allow us to
describe the set ${\rm vol}\left(N, \rm {Iso}_e\t{\rm
SL_2(\R)}\right) $ for each 3-manifold $N$ supporting an $\t{\rm
SL_2(\R)}$-geometry.

It is known that $N$ supports an $\t{\rm SL_2(\R)}$-geometry if and
only if  $N$ is a Seifert manifold with non-zero Euler number $e(N)$
over an  orbifold of negative Euler characteristic. As in \cite{EHN} we use
$\llcorner{a }\lrcorner$ and  $\ulcorner a\urcorner$ for $a\in{\R}$
to denote respectively, the greatest integer $\le a$ and the least
integer $\ge a$.

\begin{proposition}\label{Seifert volumes of Seifert manifolds}
Suppose $N$ supports the $\t{\rm SL_2(\R)}$-geometry and that its
base $2$-orbifold has a positive genus $g$. Then
\begin{eqnarray}
{\rm vol}\left(N, \rm {Iso}\t{\rm SL_2(\R)})
\right)=\left\{\frac{4\pi^2}{|e(N)|}\left(\sum_{i=1}^r\left(\frac{n_i}{a_i}\right)-n\right)^2\right\}
\end{eqnarray}
where $n_1,...,n_r,n$ are integers such that $$\sum_{i=1}^r
\llcorner{ {n_i}/{a_i}}\lrcorner -n\le 2g-2, \,\,\, \sum_{i=1}^r
\ulcorner{n_i}/{a_i}\urcorner-n \ge 2-2g$$ and $a_1, ..., a_r$ are
 the indices of the singular points of the orbifold of $N$.
\end{proposition}

\begin{remark} (1) In order to check Proposition \ref{Seifert volumes of Seifert
manifolds},  we will describe all representations with non-zero
volume (see Proposition \ref{SvSm}). They will be
 used in the  volume computations for 3-manifolds with
non-trivial geometric decompositions.

(2) Proposition \ref{Seifert volumes of Seifert manifolds} presents
explicitly the rationality of the elements in ${\rm vol}(N, \rm
{Iso}\t{\rm SL_2(\R)})$, which was proved in \cite{Re1}.
\end{remark}

\subsection{Volumes of non-geometric manifolds}

As a partial answer to Question 1 (1) for non-geometric manifolds,
it was known only recently that each non-trivial graph manifold $N$
has a finite cover $\t N$ such that ${\rm vol}(\t N, {\rm
Iso}_e\t{\rm SL_2(\R)})$ contains non-zero elements, see \cite{DW2}.
Thus Question 1 (1) (ii) is  reduced to the non-geometric
3-manifolds containing some hyperbolic pieces. In view of Theorem
\ref{basic property of volume of presentation} (2) (3), as well as
the result of \cite{DW2}, and in an attempt to seal a relation
between the \emph{Gromov simplicial volume} and the \emph{hyperbolic
volume}, Professor M. Boileau and some others wondered  the
following more direct version of Question 1 (1):

\begin{question}
Suppose $N$ has positive simplicial volume, i.e., $N$ contain some
hyperbolic geometric piece(s).

(i) Is there a representation $\rho\co\pi_1  N\to{\rm PSL}(2;{\C})$
with positive volume?

(ii) or weakly is there a representation $\rho\co\pi_1 \t N\to{\rm
PSL}(2;{\C})$ with positive volume for some finite covering $\t N$
of $N$?
\end{question}

Let $N$ be a closed irreducible non-geometric  $3$-manifold. Let $\c{T}_N\subset N$
denote the minimal union  of disjoint essential tori and
Klein bottles of $N$, unique up to isotopy, such that each piece of
$N\setminus\c{T}_N$ is either Seifert  or hyperblic (see
section 2.1). We say that $N$ is an \emph{one-edged manifold} if
$\c{T}_N$ consists of a single separating torus $T$.

 The next two propositions respectively
gives a  negative answer to Question 2 (i) and a partially positive
answer to Question 2 (ii).

\begin{proposition}\label{hyper-virt-nonzero} Let $N$ be a 3-manifold containing a hyperbolic piece $Q$ whose  each boundary component that is non-separating  in $N$ is shared by a Seifert piece of $N$.
Then there is a representation $\rho\co\pi_1\t N\to{\rm
PSL}(2;{\C})$ with positive volume for some finite covering $\t N$
of $N$.
\end{proposition}

If each component of $\c{T}_N$  is separating in $N$, then the condition in
Proposition \ref{hyper-virt-nonzero} is automatically satisfied, and
in particular

 \begin{corollary}
A rational homology sphere has a positive simplicial volume iff it admits  a
finite covering with positive hyperbolic volume.
 \end{corollary}
In the opposite direction we state:

\begin{proposition}\label{hyper-zero} There are  infinitely many 1-edged  3-manifolds
$N$   with non-vanishing $||N||$ but ${\rm vol}(N,{\rm
PSL}(2;{\C}))=\{0\}$.
\end{proposition}


By the definition of 1-edged manifold, another immediate consequence
of Proposition \ref{hyper-virt-nonzero} and Proposition
\ref{hyper-zero} is a negative answer of Question \ref{1} (2) for
hyperbolic volume, that is to say:

\begin{corollary}
The hyperbolic volume do not have the covering property. Namely
there are finite coverings $p\co \t N\to N$ such that $HV(\t
N)>|{\rm deg}p|HV(N)=0$.
 \end{corollary}

\subsection{Volumes as Chern Simons invariants}
The difficulty of Question \ref{1} more or less can be seen from the
definition: to get a non-zero element in ${\rm vol}\left(N, G
\right)$ we need first to find an \emph{a priori} "significant"
representation $\rho\co\pi_1N\to G$, and then to be able to compute
its volume. Usually non of those steps are easy, especially when the
manifold is not geometric. Basically  in the geometric case, there
is a natural significant representation given by the faithful and
discrete representation of its fundamental group in the Lie group of
its geometry.    In the non-geometric case one can use the geometry
of its  pieces and try to construct "components after components" a
global significant representation. However in this new situation
many problems occur: First the geometric pieces have non-empty
boundary and the volume of representation is not easy to manipulate
in the non-closed case and moreover we must make sure that the local
representations are compatible in the toral boundaries in order to
be glued together.    Then an other problem arises when we want to
compute the volume of a global representation from the local
volumes. Unlike the Gromov simplicial volume it is certainly
non-additive with respect to  the geometric decomposition. This
latter point is a difficulty but somehow it is also a chance that
this volume would take into account the way the geometric pieces are
glued together.   In order to prove Propositions \ref{hyper-zero}
and \ref{hyper-virt-nonzero}, as well as Propositions
\ref{Seifert-Seifert} and \ref{Seifert-hyperbolic} (see paragraph 1.4), the volume
of representations will be turned into Chern Simons invariants, with
certain semi-simple (non-compact) Lie groups, yielding computations
that were not easy before.

Denote by $G$ the semi-simple Lie group ${\rm Iso}_e\t{{\rm
SL}_2(\R)}$ or ${\rm PSL}(2,{\C})$ with the associated Riemannian
homogeneous spaces $X$ which  is $\t{{\rm SL}_2(\R)}$ or ${\Hi}^3$
endowed with the closed $G$-invariant volume form $\omega_X$.

Denote by $\mathfrak{g}$ its Lie algebra. We recall (see Section 5
for more details) that the Chern Simons classes with structure group
${\rm PSL}(2,{\C})$ are based on the first Pontrjagin class and in
the same way we define the Chern Simons classes with structure group
${\rm Iso}_e\t{{\rm SL}_2(\R)}$  based on the invariant polynomial
defined by ${\bf R}(A\otimes A)={\rm Tr}(X^2)+t^2$ where $A$ is an
element of the Lie algebra of ${\rm Iso}_e\t{{\rm SL}_2(\R)}$ which
decomposes into $X+t$ where $X$ is in the Lie algebra of $\t{{\rm
SL}_2(\R)}$ and $t\in{\R}$.

\begin{proposition}\label{vol} Let $\rho$ be a  representation of $\pi_1N$ into $G={\rm Iso}_e\t{{\rm SL}_2(\R)}$
and let   $A$ be a  corresponding flat $G$-connection in the
principal bundle $P=N\times_\rho G$. If $P$ admits a section
$\delta$ over $N$ then
\begin{eqnarray}
\mathfrak{cs}_N(A,\delta)=\int_N\delta^{\ast}{\bf R}\left(dA\wedge
A+\frac{1}{3}A\wedge[A,A]\right)=\frac{2}{3}{\rm vol}_G(N,\rho)
\end{eqnarray}
In particular the Chern-Simons invariant of flat ${\rm Iso}_e\t{{\rm
SL}_2(\R)}$-connections is gauge invariant.
\end{proposition}
\begin{remark}
Assuming that $P=M\times_\rho G$ admits a section means equivalently
that $\rho$ admits a lift into $\t{{\rm SL}_2(\R)}$ so that the
bundle admits a reduction to an $\t{{\rm SL}_2(\R)}$-bundle and we
reckon that the correspondence in Proposition \ref{vol} for
$G=\t{{\rm SL}_2(\R)}$  is pointed in \cite{Re1}, and verified in
\cite{Kh} by  a long and subtle computation. However for our own
understanding we reprove it in a very simple way  underscoring that
the correspondence is quite natural and comes directly from the
structural equations of the Lie group involved (see Section 5.4).
\end{remark}
The following correspondence is derived from \cite{KK}. Let's denote
the imaginary part of the complex number $z$ by $\Im(z)$.
\begin{proposition}\label{volh} Let $\rho$ be a  representation of $\pi_1N$ into $G={\rm PSL}(2;{\C})$ and let   $A$ be a  corresponding flat $G$-connection over $N$. If
$N\times_\rho G$ admits a section $\delta$ over $N$ then
\begin{eqnarray}
\Im(\left(\mathfrak{cs}_N(A,\delta)\right)=-\frac{1}{\pi^2}{\rm
vol}_{G}(N,\rho)
\end{eqnarray}
\end{proposition}

\begin{remark}\label{volhh}
The imaginary part of the Chern Simons invariants of flat ${\rm
PSL}(2;{\C})$-connections is gauge invariant from the formula for it
does not depend on the chosen section. We don't give any geometric
interpretations for the real part of  $\mathfrak{cs}_M(A,\delta)$. It
is not gauge invariant and it won't be used throughout our proofs.
If the developing map $D_\rho \co\t{M}\to{\Hi}^3$, corresponding to
$\rho\co\pi_1M\to G $, were an isometry with respect to the
pull-backed metric of ${\Hi}^3$, then certainly 
$\Re(\mathfrak{cs}_M(A,\delta))$ would be  the ${\R}/{\Z}$-valued
Chern-Simons invariant of the Levi Civita connection corresponding
to the  Riemannian metric in $\t{M}$ pull-backed from the hyperbolic
metric by $D^{\ast}$. This is  correct  when $M$ is itself a
complete hyperbolic manifold and when $\rho$ is a faithful discrete
representation, as it is stated in \cite{KK}.
\end{remark}


\subsection{Complexity of the sewing involution}
Recall that each 3-manifold with non-trivial geometric decomposition
is determined by the topology of its geometric pieces and the
isotopy class of the gluing maps among them. At this point it is
worth recalling that the simplicial volume $||*||$ tells nothing
about the gluing. However both ${\rm vol}\left(N, {\rm
PSL}(2;{\C})\right)$ and ${\rm vol}(N, \rm {Iso}\t{\rm SL_2(\R)})$
somehow do contain the gluing information.

To illustrate this fact and to avoid complicated computations,
below, we often focus on the simplest 3-manifolds with non-trivial
geometric decomposition, namely the $1$-edged manifolds.
We present them as
$N=Q_-\cup_\tau Q_+$, where $\tau\co\b Q_-=T_-\to \b Q_+=T_+$ is the
gluing map of their two geometric pieces $Q_-$ and $Q_+$.

Recall that $N$ is termed a \emph{graph manifold} if both $Q_-$ and $Q_+$ are Seifert. On the other hand a finite covering $p\co\t{N}\to N$ is termed \emph{$q\times q$-characteristic}, for some integer $q$, if for any component $\t{T}$ over $T$ the map $p$ induces the covering $p|\co\t{T}\to T$ corresponding to the subgroup $q{\Z}\oplus q{\Z}$ of ${\Z}\oplus{\Z}=\pi_1T$.

We fix a basis $s_\e,h_\e$ for $H_1(\b Q_\e;{\Z})$, $\e=\pm$, so that
the isotopy class of $\tau$ is given by an integral matrix
$\tau_*=\begin{pmatrix}
a&b \\
c&d
\end{pmatrix}$ with  determinant $-1$ and $$\tau(s_-)=as_++ch_+\  {\rm and} \  \tau(h_-)=bs_++dh_+.$$
Moreover, we choose the prefered basis:

(A) when $Q_\e$ is a Seifert manifold,  we require that
$T_{\e}(s_\e,h_\e)$ is a section-fiber basis (see section 2.3 for the definition). Notice that when both
$Q_-$ and $Q_+$ are Seifert manifolds, then $b\not=0$, otherwise $N$
would be a Seifert manifold.

(B) when $Q_\e$ is a hyperbolic piece, we require that
$T_{\e}(s_\e,h_\e)$ is a basis that consists of the first and second
shortest simple closed geodesic on its Euclidean boundary on the maximal
cusp (see section 2.2). Under this basis, the norm of a curve $as_\e+bh_\e$ defined by $\sqrt{a^2+b^2}$ is
equivalent to the Euclidean norm in the cusp. Let us denote by $0<k\leq K$ the real constants such that for any $a,b\in{\Z}$ then
\begin{eqnarray}
k\sqrt{a^2+b^2}\leq  {\rm length}(as_{\e}+bh_{\e})\leq
K\sqrt{a^2+b^2}
\end{eqnarray}
See Remark \ref{universal constant} for more details about this construction.

\begin{proposition}\label{Seifert-Seifert} Let $N=Q_-\cup_\tau Q_+$ be an  one-edged  graph manifold and denote by $G$  the group ${\rm Iso}_e\t{{\rm
SL}_2(\R)}$. There exists a $n$-fold $q\times q$-characteristic
covering $\t{N}\to N$, where $n,q$ depend only on $Q_-$ and  $Q_+$,
and a representation $\varphi\co\pi_1\t{N}\to G$  such that

 $${\rm vol}_G(\t{N},\varphi)=8\pi^2\frac{n}{q^2}\   {\rm if}\  a=d=0,\ \    {\rm vol}_G(\t{N},\varphi)=4\pi^2\frac{n}{q^2|b|}\   {\rm if}\   c=0,\ ,$$

 $${\rm vol}_G(\t{N},\varphi)=4\pi^2\frac{n}{q^2|ac|}\   {\rm if}\     ac\not=0,\ \    {\rm vol}_G(\t{N},\varphi)=4\pi^2\frac{n}{q^2|cd|}\   {\rm if}\      cd\not=0.$$
\end{proposition}
There exists a \emph{uniform constant} $C=2\pi/k$ such that for any one-cusped hyperbolic $3$-manifold $Q$ then a  deformation of the
hyperbolic structure on $Q$ can be extended to a complete hyperbolic
one in the surgered manifold $Q(a,b)$ provided $\|(a,b)\|_2>C$ where $a,b$ are the
co-prime coefficients of the curve of $\b Q$ (in the chosen basis)
identified with the meridian of the solid torus (see paragraph 2.2).
\begin{proposition}\label{Seifert-hyperbolic} Let $N=Q_-\cup_\tau Q_+$ be an   one-edged   $3$-manifold, where $Q_-$ is Seifert and $Q_+$ is hyperbolic.
Then  there exists  a $n$-fold $q\times q$-characteristic covering
$\t{N}\to N$, where $n,q$ depend only on $Q_-$ and  $Q_+$,  and a
representation $\varphi\co\pi_1\t{N}\to{\rm PSL}(2;{\C})$ such that  for any
$\|(a,c)\|_2>C$ then
\begin{eqnarray}
{\rm vol}_{{\rm PSL}(2;{\C})}(\t{N},\varphi)=n{\rm vol}Q_+(a,c)+\frac{\pi n(q-1)}{2q}{\rm length}(\gamma)
\end{eqnarray}
where $\gamma$ is the geodesic added to $Q_+$ to complete the cusp
with respect to the $(a,c)$-Dehn filling and ${\rm length}(\gamma)$ denotes its length in the complete hyperbolic structure of $Q_+(a,c)$. The same statement is true
for $(b,d)$.
\end{proposition}

\begin{remark}\label{zag}
  By the computations
made in \cite{NZ}  we get
\begin{eqnarray}
{\rm vol}Q_+(a,c)={\rm vol}Q_+-\frac{\pi}{2}{\rm length}(\gamma)+O\left(\frac{1}{a^4+c^4}\right)
\end{eqnarray}
where
\begin{eqnarray}
 {\rm length}(\gamma)=2\pi\frac{\Im{(z_0)} }{|a+z_0.c|^2}+O\left(\frac{1}{a^4+c^4}\right)
 \end{eqnarray}
where $z_0$ is a complex number with $\Im{(z_0)}>0$ giving the
modulus of the Euclidean structure on the torus $T$ corresponding to
the cusp of $Q_+$. Substituting (1.6) and (1.7) into Proposition
\ref{Seifert-hyperbolic} we have by \cite[Theorem
1A]{NZ}
\begin{eqnarray}
\frac{{\rm vol}_{{\rm PSL}(2;{\C})}(\t{N},\varphi)}{n}={\rm vol}Q_+-\pi^2\frac{\Im{z_0}}{q|a+z_0c|^2}+O\left(\frac{1}{a^4+c^4}\right)
\end{eqnarray}
\end{remark}

\subsection{Organization of the paper } The remaining of the paper is reflected from table of  contents.
Sections 2, 3, and 5 present necessary background and results from
3-manifold theory, volume of representations, and Chern-Simons
theory respectively. The efforts are made in organizing  those
materials  so that our results can be verified smoothly, and readers
can access the topic easily. We will verify Propositions 1.4  in
section 4,
 Proposition 1.10, 1.12  in sub-sections 5.4 and 5.5, Proposition 1.6 in Section 6,  Proposition  1.8 in Sections 7, and   Propositions   1.14, 1.15 in section 8 respectively.

\bigskip\noindent\textbf{Acknowledgement}. The first author would like to thank the  Department  of Mathematics of the Peking University for its support and  its hospitality during the elaboration of this paper. The second author is partially supported by grant
No.10631060 of the National Natural Science Foundation of China. We
thank  Professor Michel Boileau and Professor Daniel Matignon for
helpful communications.

 \section{Topology and geometry of 3-manifolds}

\subsection{Thurston's picture of 3-manifolds}
Each closed orientable 3-manifold $N$ admits a unique prime decomposition
$N_1\#.....\#N_k$, the prime factors being unique up to the order and
up to homeomorphisms.

Let $N$ be a compact orientable 3-manifold. Call an embedded surface
$T$ in $N$ \emph{essential} if $T$ is incompressible (see \cite[page
23]{Ja}) and is not parallel to any component of $\partial N$.

Let $N$ be a closed orientable prime 3-manifold. According to the
theory of Thurston, Johannson and  Jaco-Shalen (\cite{Th1},
\cite{Th2}, \cite{JS},  \cite{Joh}) combined with the geometrization
of $3$-manifolds achieved by  G. Perelman and W.
Thurston,  then either

(i) $N$ supports one of the following eight geometries: ${\Hi}^3$,
$\widetilde{{\rm SL}_2({\R})}$, ${\Hi}^2\times{\R}$, ${\rm Sol}$, ${\rm Nil}$, ${\R}^3$, ${\S}^3$ and
${\S}^2\times {\R}$ (where ${\Hi}^n$, ${\R}^n$ and ${\S}^n$ are the $n$-dimensional
hyperbolic space, Euclidean space and the sphere respectively); in these
cases  $N$ is called geometric; or

(ii) there is a minimal union $\c{T}_N\subset N$ of disjoint
essential tori and Klein bottles of $N$, unique up to isotopy, such
that each piece of $N\setminus\c{T}_N$ is either a Seifert fibered manifold,
supporting the ${\Hi}^2\times{\R}$-geometry,  or the ${\Hi}^3$-geometry. We call  the components of $N\setminus\c{T}_N$ the geometric pieces (Seifert pieces and hyperbolic pieces respectively).

Call a prime closed orientable 3-manifold $N$ a {\it (non-trivial) graph
manifold} if $N$ has a (non-trivial) geometric decomposition but
contains no hyperbolic pieces.

\subsection{Thurston Hyperbolic Dehn filling Theorem }\label{bh}

We denote by ${\Hi}^n$ the $n$-dimensional hyperbolic space that can
be seen as the half-space
${\R}^{n-1}\times(0,\infty)\subset{\R}^{n-1}\times{\R}$ endowed with
the metric $(dx^2_1+...+dx^2_{n})/x_n^2$. The orientations
preserving isometry group ${\rm Iso}_+({\Hi}^n)$ of ${\Hi}^n$ can be
metrically thought of as the oriented unit frame bundle over  ${\Hi}^n$ and
algebraically it is identified with ${\rm PSL}(2;{\R})$ when $n=2$
and ${\rm PSL}(2;{\C})$ when $n=3$.

Let $M$ denote a compact, orientable $3$-manifold whose boundary
consists of tori $T_1,...,T_p$ and whose interior admits a complete
(finite volume) hyperbolic structure. In other words, ${\rm int}M$
is isometric to ${\Hi}^3/\Gamma$, where $\Gamma$ is a discrete,
torsion free subgroup of  ${\rm PSL}(2;{\C})$.
For each $T_i\subset\b M$ we fix a basis $\mu_i,\lambda_i$ of
$H_1(T_i;{\Z})$. We define the surgered manifold $M((a_1,b_1),...,(a_p,b_p))$
as follows:

(1) if $(a_i,b_i)=\infty$ then $T_i$ is left unfilled,

(2) if $(a_i,b_i)$ are coprime  then we perform a $(a_i,b_i)$-Dehn
filling on $T_i$ by gluing the solid torus $V={\bf
D}^2\times{\S}^1$, with basis $(m,l)$ so that the meridian $m$ is
identified with the isotopy class of the simple closed curve
$a_i\mu_i+b_i\lambda_i$ in $T_i$,

(3) otherwise, denote by $q_i>1$ the greatest common divisor of
$a_i$ and $b_i$ so that $a_i=q_ir_i$ and $b_i=q_is_i$ with
$(r_i,s_i)$ coprime. We denote by $\c{R}_{q_i}$ the rotation of
order $q_i$ centered at $0$ in ${\R}^2$, by ${\bf D}^2$ the closed unit
$2$-disk and by ${\bf D}^2/\c{R}_{q_i}$ the singular disk under the
action of the group generated by  $\c{R}_{q_i}$. Then we glue the
singular solid torus $V(q_i)={\bf D}^2/\c{R}_{q_i}\times{\S}^1$ so
that its meridian $m$ is identified with the isotopy class of the
simple closed curve $r_i\mu_i+s_i\lambda_i$ in $T_i$. Let's recall
the Thurston's Hyperbolic Dehn surgery theorem, \cite[5.8.2]{Th1}
(see also \cite{BMP} and \cite{DuM}).
 \begin{theorem}\label{surgered}
Let $M$ be compact oriented $3$-manifold  with  toral boundary
$T_1\cup...\cup T_p=\b M$ whose interior admits a complete
hyperbolic structure. Then there is a real number $C>0$ such that
if $\|(a_i,b_i)\|_2> C$ for $i=1,...,p$, then the
surgered space $M((a_1,b_1),...,(a_p,b_p))$ is a complete hyperbolic
orbifold.
 \end{theorem}
 The complete hyperbolic metric on $M$ corresponds to
a faithful, discrete representation $\rho_{d_0}\co\pi_1M\to{\rm
PSL}(2;{\C})$ such that, up to conjugation,
$\rho_{d_0}(\mu_i)=\begin{pmatrix}
1&1 \\
0&1
\end{pmatrix}$ and $\rho_{d_0}(\lambda_i)=\begin{pmatrix}
1&\eta_i \\
0&1
\end{pmatrix}$.
Theorem \ref{surgered} claims that if each $(a_i,b_i)$ satisfies
$\sqrt{a^2_i+b_i^2}>C$, for $i=1,...,p$, then the former
representation  may be modified such that, up to conjugation,

 $\rho_d(\mu_i)=\begin{pmatrix}
e^{2i\pi\alpha_i}&0 \\
0&e^{-2i\pi\alpha_i}
\end{pmatrix}$ and $\rho_d(\lambda_i)=\begin{pmatrix}
e^{2i\pi\beta_i}&0 \\
0&e^{-2i\pi\beta_i}
\end{pmatrix}$
 and the induced structure on ${\rm
int}M$ is a non-complete hyperbolic metric that can be completed in
the surgered hyperbolic orbifold  $M((a_1,b_1),...,(a_p,b_p))$. Moreover Thurston's
Theorem shows that if there exists $i$ such that
$(a_i,b_i)\not=\infty$ then $${\rm vol}_{{\Hi}^3}M((a_1,b_1),...,(a_p,b_p))<{\rm
vol}_{{\Hi}^3}M$$ and $$\lim_{(a_1,b_1)\to\infty,...,(a_p,b_p)\to\infty}{\rm
vol}_{{\Hi}^3}M((a_1,b_1),...,(a_p,b_p))={\rm vol}_{{\Hi}^3}M$$

\begin{remark}\label{universal constant} We denote by $M_\text{max}$ the interior of $M$ with a
system of maximal cusps removed. Now we identify $M$ with
$M_\text{max}$, then $\partial M$ has a Euclidean metric induced
from the hyperbolic metric and each closed Euclidean geodesic in
$\partial M$ has the induced length. The so called $2\pi$-lemma
claims that  $M((a_1,b_1),...,(a_p,b_p))$ is hyperbolic if  the geodesic
corresponding to $(a_i, b_i)$ has length $> 2\pi$ for $i=1,...,p$ (see \cite{BH} for
example). The first and the second shortest simple geodesics on each
component $T_i$ of $\b M$ must form a basis, and under this basis the norm
of a curve $(a, b)$ defined by $\sqrt{a^2+b^2}$ is equivalent to the
Euclidean length in $T_i$ (\cite[see page 309]{HRW} for example).
So  under such a basis, there is a universal constant $C$ such that,
for any  one cusped hyperbolic manifold, if $ \sqrt{a_i^2+
b_i^2}>C$, then $M(a_i, b_i)$ is hyperbolic.
\end{remark}

 \subsection{Seifert geometry}

We consider the group ${\rm PSL}(2;{\R})$  as the orientation
preserving isometries of the hyperbolic
  $2$-space ${\Hi}^2=\{z\in{\C}, \Im(z)>0\}$ with $i$ as a base point. In this way ${\rm PSL}(2;{\R})$ is a (topologically trivial) circle bundle over ${\Hi}^2$.
  Denote by $p\co\t{{\rm SL}_2(\R)}\to{\rm PSL}(2;{\R})$ the universal covering of ${\rm PSL}(2;{\R})$ with the induced metric. Then $\t{{\rm SL}_2(\R)}$
  is a (topologically trivial) line bundle over ${\Hi}^2$. For any $\alpha\in{\R}$, denote by ${\rm sh}(\alpha)$ the element of $\t{{\rm SL}_2(\R)}$
  whose projection into ${\rm PSL}(2;{\R})$ is given by $\begin{pmatrix}
\cos(2\pi\alpha)&\sin(2\pi\alpha) \\
-\sin(2\pi\alpha)&\cos(2\pi\alpha)
\end{pmatrix}$. Then the set $\{{\rm sh}(\alpha), n\in{\Z}\}$,
  is the kernel of $p$ as well as the center of $\t{{\rm SL}_2(\R)}$, acting by integral translation along the fibers of $\t{{\rm SL}_2(\R)}$.
  By extending this ${\Z}$-action on the fibers to the ${\R}$-action we get the whole identity component of the isometry group of $\t{{\rm SL}_2(\R)}$.
  To summarize we have the following diagram of central extensions
$$\xymatrix{
\{0\} \ar[r] \ar[d] & {\Z} \ar[r] \ar[d] & \t{{\rm SL}_2(\R)}\ar[r] \ar[d] & {\rm PSL}(2;{\R})\ar[r] \ar[d] & \{1\} \ar[d] \\
\{0\} \ar[r]  & {\R} \ar[r]  & {\rm Iso}_e(\t{{\rm SL}_2(\R)})
\ar[r]  & {\rm PSL}(2;{\R})\ar[r]  & \{1\} }$$ In particular the
group ${\rm Iso}_e(\t{{\rm SL}_2(\R)})$ is generated by $\t{{\rm
SL}_2(\R)}$ and the image of ${\R}$ which intersect together in the
image of ${\Z}$.

\begin{remark}\label{SL}
The Lie group  ${\rm Iso}_e(\t{{\rm SL}_2(\R)})$ can be thought of as a
quotient of the direct product ${\R}\times\t{{\rm SL}_2(\R)}$  -
(where each element $x$ on ${\R}$ is naturally identified with the
translation $\tau_x$ of length $x$) - under the relation
$(x,h)\sim({x'},h')$ if and only if there exists an integer
$n\in{\Z}$ such that ${x'}-x=n$ and $h'={\rm sh}(-n)\circ h$. In
this way ${\rm Iso}_e(\t{{\rm SL}_2(\R)})$ can be identified with
${\R}\times_{\Z}\t{{\rm SL}_2(\R)}$ and it is easily checked that
this group is homotopic to the circle.
\end{remark}

 Let  $F_{g,n}$ be an
oriented  $n$-punctured surface of genus $g\geq 0$ with boundary
components $s_1,...,s_n$ with $n\ge 0$. Then
$N'=F_{g,n}\times{\S}^1$ is oriented if ${\S}^1$ is oriented. Let
$h_i$ be the oriented ${\S}^1$-fiber on the torus $T_i=s_i\times h_i$. We say that $(s_i,h_i)$ is a \emph{section-fiber}  basis of $T_i$. Let $0\leq s\leq n$. Now attach $s$ solid tori $V_i$'s to the
boundary tori $T_i$'s of $N'$ such that the meridian of $V_i$ is identified
with the slope ${a_i}s_i+{b_i}h_i$ where $a_i>0,
(a_i,b_i)=1$ for $i=1,...,s$. Denote the resulting manifold
by $\left(
g,n-s;\frac{b_1}{a_1},\cdots,\frac{b_s}{a_s}\right)$
which has the Seifert fiber structure extended from the circle
bundle structure of $N'$. Each orientable Seifert fibered space with
orientable base $F_{g,n-s}$ and with $\leq s$ exceptional fibers is
obtained in such a way.
If $N$ is closed, i.e. if $s=n$, then  define the Euler number of
the Seifert fibration by
$$e(N)=\sum_1^s \frac{b_i}{a_i}\in\mathbb{Q}$$
and  the Euler characteristic of the orbifold $O(N)$ by
$$\chi_{O(N)}=2-2g-\sum_{i=1}^s\left(1-\frac 1{a_i}\right)\in\mathbb{Q}.$$

From \cite{BG2} we know that a closed orientable 3-manifold $N$ supports the $\widetilde {{\rm
SL}_2}$-geometry, i.e.  there is a discrete and faithful representation
$\psi: \pi_1N\to {\rm Iso} \widetilde {{\rm SL}_2}$, if and only if
$N$
 is a Seifert manifold with non-zero Euler number $e(N)$ and
negative Euler characteristic $\chi_{O(N)}$.

\section{Volume of representations of closed manifolds}
Given a  semi-simple, connected Lie  group $G$ and a closed oriented
manifold $M^n$ of the same dimension than  the contractible space
$X^n=G/K$, where $K$ is a maximal compact subgroup of $G$, we can
associate, to each representation $\rho\co\pi_1M\to G$, a volume
${\rm vol}_G(M,\rho)$ in the following ways.

\subsection{Volume of representations}
 First of all, fix a
$G$-invariant Riemannian metric $g_X$ on $X$, then denote by
$\omega_X$ the corresponding $G$-invariant volume form. We think of the elements $\t{x}$ of  $\t{M}$
as the homotopy classes of paths $\gamma\co[0,1]\to M$ with $\gamma(0)=x_{0}$ which are acted by $\pi_1(M,x_0)$ by setting $[\sigma].\t{x}=[\sigma.\gamma]$, where $.$ denotes the paths composition.

A developing map $D_{\rho}\co\t{M}\to X$ associated to $\rho$ is a
$\pi_1M$-equivariant map  from the universal covering $\t{M}$ of
$M$, acted by $\pi_1M$, to $X$, endowed
with the $\pi_1M$ action induced by $\rho\co\pi_1M\to G$. That is to say:
for any $x\in \t{M}$ and $\alpha \in \pi_1M$, then

$$D_{\rho}(\alpha.x)=\rho(\alpha)^{-1}D_{\rho}(x)$$

Such a map does exist and can be constructed explicitly as in
\cite{BCG}: Fix a triangulation $\Delta_M$ of $M$. Then its lift is a
triangulation $\Delta_{\t{M}}$ of $\t{M}$, which is $\pi_1M$-equivariant.
Then fix a fundamental domain $\Omega$ of $M$ in $\t{M}$ such that
the zero skeleton $\Delta^0_{\t{M}}$ misses the frontier of $\Omega$. Let
$\{x_1,...,x_l\}$ be  the vertices of $\Delta^0_{\t{M}}$ in $\Omega$, and
let $\{y_1,...,y_l\}$ be  any $l$ points in $X$.  We first set

$$D_{\rho}(x_i)=y_i, \,\, i=1 ,..., l.$$

Next extend $D_{\rho}$ in an $\pi_1M$-equivariant way to $\Delta^0_{\t{M}}$:
For any vertex $x$ in $\Delta^0_{\t{M}}$, there is a unique  vertex
$x_i$ in $\Omega$ and $\alpha_x\in \pi_1M$ such that
$\alpha_x.x_i=x$, and we set
$D_{\rho}(x)=\rho(\alpha_x)^{-1}D_{\rho}(x_i)$. Finally we extend
$D_{\rho}$ to edges, faces, ..., and $n$-simplices of $\Delta_{\t{M}}$
by straightening the images to geodesics using the homogeneous
metric on the contractible space $X$.
This  map is unique up to equivariant homotopy. Then
$D_{\rho}^{\ast}(\omega_X)$ is a $\pi_1M$-invariant closed $n$-form
on $\t{M}$ and therefore can be thought of as a closed $n$-form on $M$. Thus define

$${\rm vol}_G(M,\rho)=\left|\int_MD_{\rho}^{\ast}(\omega_X)\right|=\left|\sum_{i=1}^s\epsilon_i {\rm vol}_X(D_{\rho}(\t\Delta_i))\right|$$
where $\{\Delta_1,...,\Delta_s\}$ are the $n$-simplices of $\Delta_M$,
$\t \Delta_i$ is a lift of $\Delta_i$ and $\epsilon_i=\pm 1$ depending on whether $D_{\rho}|\t\Delta_i$ is preserving  or reversing orientation.

\subsection{Volume of representations as a continuous cohomology class}\label{contcoh}
 Let
$o=\{K\}$ be the base point of $X=G/K$ and for any $g_1,...,g_l\in G$
denote by $\Delta(g_1,...,g_l)$ the geodesic $l$-simplex of $X$ with
vertices $\{o,g_1(o),...,g_l...g_2g_1(o)\}$.   There is a natural
homomorphism
$$H^{\ast}(\mathfrak{g},\mathfrak{k};{\R})=H^{\ast}(G{\rm -invariant\ differential\ forms\ on\ }X)\to H^{\ast}_{\rm cont}(G;{\R})$$
defined in \cite{Du}  by $\eta\mapsto
\left((g_1,...,g_l)\to\int_{\Delta(g_1,...,g_l)}\eta\right)$ which
turns out to be an isomorphism by the Van-Est Theorem \cite{V}.

Recall that for each representation $\rho\co\pi_1M\to G$ one can
associate a flat bundle over
 $M$ with fiber $X$ and group $G$ constructed as follows: $\pi_1M$ acts diagonally on the product
 $\t{M}\times X$  by the following formula
 \begin{eqnarray}
\sigma.(\t{x},g)=(\sigma.\t{x},\rho^{-1}(\sigma)g)
 \end{eqnarray}
 and we can form the quotient $\t{M}\times_\rho X=\t{M}\times X/\pi_1M$ which is the flat $X$-bundle over
 $M$ corresponding to $\rho$.

Then  for each $G$-invariant closed form $\omega$ on $X$,
$q^*(\omega)$ is a $\pi_1(M)$-invariant closed form on $\t{M}\times
X$, where $q: \t{ M}\times X\to X$ is the projection, which induces
a form $\omega'$ on $M\times _\rho X$. Then $s^*(\omega')$ is a
closed form on $M$, where $s: M\to M\times _\rho X$ is a section (since $X$ is
contractible, the sections exist and are  all  
homotopic).
Thus any representation  $\rho\co\pi_1M\to G$ leads to a natural homomorphism
$$\rho^*: H^{\ast}_{\rm cont}(G;{\R})=H^{\ast}(G{\rm -invariant\ differential\ forms\ on\ }X)\to  H^{\ast}(M;{\R})$$
induced by $\rho^*(\omega)=s^*{\omega}'$. The volume of $\rho$ is therefore defined by ${\rm
vol}_G(M,\rho)=\left|\int_M\rho^{\ast}(\omega_X)\right|$.

  The equivalence between the two definitions is immediate since the
$\pi_1M$-equivariant map ${\rm Id\times D_\rho}\co\t{M}\to \t
M\times X$ descends to a section $M\to M\times _\rho X$.

  When the situation is clear from the context,  we drop the reference to the structural group
  $G$ in the notation ${\rm vol}_G(M,\rho)$ and we denote it by ${\rm vol}(M,\rho)$.

  \subsection{Volume of representations via transversely projective foliations}\label{foliation}

Let $\mathfrak{F}$ be a co-dimension one foliation on a closed
smooth manifold $M$ determined by a 1-form $\omega$. Then by the
Froebenius Theorem one has $d\omega= \omega\wedge\delta$ for some
1-form $\delta$. It was observed by Godbillon and Vey \cite{GV} that
the 3-form $\delta \wedge
 d\delta$ is closed and the class $[\delta \wedge
 d\delta]\in H^3(M,{\R})$ depends only on the foliation $\mathfrak{F}$ (and not on the chosen form $\omega$). This cohomology class is termed
\emph{the Godbillon-Vey class} of the foliation  $\mathfrak{F}$ and denoted by
 $GV(\mathfrak{F})$.

 \begin{proposition}\label{Euler} (\cite[Proposition 1]{BG1}) Suppose $\mathfrak{F}$
  is a horizontal flat structure  on a circle  bundle ${\S}^1\to E \to M$ with structural group ${\rm PSL}_2(\R)$. Then
 $$\int_{\S^1} GV(\mathfrak{F})=4\pi^2\t{e}(E)$$
 where $\int_{\S^1}\co H^3(E)\to H^2(M)$ denotes the integration along the fiber and $\t{e}$ denotes the Euler class of the bundle.
 \end{proposition}

Let $M$ be a closed orientable 3-manifold and $\phi\co\pi_1M\to {\rm
PSL}_2(\R)$ be a representation with zero Euler class. Since ${\rm
PSL}_2(\R)$ acts on ${\S}^1$ then one can consider the corresponding
flat circle bundle $M\times_{\phi}{\S}^1$ over $M$ and the
associated horizontal $({\rm PSL}_2(\R),{\S}^1)$-foliation
$\mathfrak{F}_{\phi}$. Since the Euler class of $\phi$ is zero we
can choose a section $\delta $ of $M\times_{\phi}{\S}^1\to M$.
Brooks and Goldman (see \cite[Lemma 2]{BG1}) showed that
$\delta^{\ast}GV(\mathfrak{F}_{\phi})$ only depends on $\phi$ (and
not on a chosen section $\delta$). Then they  defined  the \emph{Godbillon Vey} invariant of
$\phi$ by setting
$$GV(\phi)=\int_M\delta^{\ast}GV(\mathfrak{F}_{\phi})$$

\begin{remark}
Assume that $M$ is already endowed with a $({\rm
PSL}_2(\R),{\S}^1)$-foliation $\mathfrak{F}$. Then there is a
canonical way to define a flat ${\S}^1$-bundle $E$ over $M$ with
structure group ${\rm PSL}_2(\R)$, a horizontal foliation
$\mathfrak{F}_\phi$ on $E$, and a section $s: M\to E$ such that
$\mathfrak{F}=s^*\mathfrak{F}_\phi$.
  Then if $\phi$   denotes associated representation   we get (\cite{BG1} Lemma 1)
 $$GV(\phi)=\int_MGV(\mathfrak{F})$$
\end{remark}

For a given representation $\phi\co\pi_1M \to{\rm PSL}_2(\R)$, $\phi$
lifts to $\t\phi\co\pi_1M \to{\t{{\rm SL}_2(\R)}}$ if and only if
$\t{e}(\phi)=0$ in $H^2(M,{\Z})$.
The following fact has been verified in \cite{BG1}.

\begin{proposition}\label{CV=Vol} Let $M$ be a closed oriented
$3$-manifold, let $\phi\co\pi_1M\to {\rm PSL}_2(\R)$ be a representation
with zero Euler class  and fix a lift  $\t\phi\co\pi_1M \to
 \t{{\rm SL}_2(\R)}$ of $\phi$. Then
$$GV(\phi)={\rm vol}_{\t{{\rm SL}_2(\R)}}(M,\t\phi)$$
where $\t{{\rm SL}_2(\R)}$ is viewed as a semi-simple Lie group acting on itself by multiplication with corresponding homogeneous space $\t{{\rm SL}_2(\R)}$.
\end{proposition}

\section{Seifert  volume of representations of Seifert manifolds}
This section is devoted to the proof of Proposition \ref{Seifert volumes of Seifert manifolds}.
Let $N$ be a closed oriented $\t{{\rm SL}_2(\R)}$-manifold whose
base $2$-orbifold is an orientable  hyperbolic
$2$-orbifold $\c{O}$ with positive genus $g$ and  $p$ singular points. Then, keeping the same notation as in section 2.3, we have a
presentation
$$\pi_1N=\l \alpha_1,\beta_1,...,\alpha_g,\beta_g,s_1,...,s_p,h :  $$
$$s_1^{a_1}h^{b_1}=1,..., s_p^{a_p}h^{b_p}=1, [\alpha_1,\beta_1]...[\alpha_g,\beta_g]=s_1....s_p\r$$
with the condition $e=\sum_ib_i/a_i\not=0$.
The following result
of Eisenbud-Hirsch-Neumann \cite{EHN}, which extends the result of
Milnor-Wood's (see \cite{Mi} and \cite{Wo}) from circle bundles to Seifert manifolds, is very
useful for our purpose.

\begin{theorem}\label{Eisenbud-Hirsch-Neumann} Suppose $N$ is a closed orientable Seifert manifold with
a regular fiber $h$ and  base of genus $>0$. Then

(1) (\cite[Corollary 4.3]{EHN}) There is a $({\rm PSL}_2\R,{\S}^1)$-horizontal foliation on $N$ if and only if there is a representation
$\t{\phi} : \pi_1(N)\to \widetilde{{\rm SL}_2\R}$ such that
$\t{\phi} (h)= {\rm sh}(1)$;

(2) (\cite[Theorem 3.2 and Corollary 4.3]{EHN}) Suppose $N=(g,0;
a_1/b_1,..., a_n/b_n)$, then there is a $({\rm
PSL}_2\R,{\S}^1)$ horizontal foliation on $N$ if and only if

$$\sum \llcorner{ {b_i}/{a_i}}\lrcorner \le -\chi(F_g); \,\,\, \sum \ulcorner{b_i}/{a_i}\urcorner \ge
\chi(F_g),$$ 
\end{theorem}

In order to prove Proposition \ref{Seifert volumes of Seifert
manifolds} we will check the following Proposition which describes
those representations leading to a non zero volume. For each element
$(a, b)\in {\R}\times\t{{\rm SL}_2(\R)}$, its image in
${\R}\times_{\Z}\t{{\rm SL}_2(\R)}$ will be denoted as $\o{(a, b)}$.

\begin{proposition}\label{SvSm}
A representation $\rho: \pi_1(N)\to {\rm Iso}_e\t{{\rm
SL}_2(\R)}={\R}\times_{\Z}\t{{\rm SL}_2(\R)}$ has non-zero volume iff
there are integers $n,n_1,...,n_p$ subject to the conditions
\begin{eqnarray}
\sum \llcorner{ {n_i}/{a_i}}\lrcorner -n\le 2g-2 \,\,\, {\rm and}  \,\,\,\sum \ulcorner{n_i}/{a_i}\urcorner-n \ge
2-2g
\end{eqnarray}
such that
\begin{eqnarray}
\rho(s_i)=\o{\left(\frac{n_i}{a_i}-\frac{b_i}{a_i}\frac{1}{e}\left(\sum_i\left(\frac{n_i}{a_i}\right)-n\right),g_i{\rm sh}\left(\frac{-n_i}{a_i}\right)g^{-1}_i\right)}
\end{eqnarray}
where $g_i$ is an element of $\t{{\rm SL}_2(\R)}$ and
\begin{eqnarray}
\rho(h)=\o{\left(\frac{1}{e}\left(\sum_i\left(\frac{n_i}{a_i}\right)-n\right),1\right)}
\end{eqnarray}
whose volume is given by
\begin{eqnarray}
{\rm
vol}(N,\rho)=4\pi^2\frac{1}{|e|}\left(\sum_i\left(\frac{n_i}{a_i}\right)-n\right)^2
\end{eqnarray}
Moreover the $\rho$-image of $\alpha_1,\beta_1, ... ,
\alpha_g,\beta_g$ can be chosen to lie in $\t{{\rm SL}_2(\R)}$.
\end{proposition}

\begin{proof}
The condition ${\rm vol}(N,\rho)\ne 0$ implies that $\rho(h)=
\o{(\zeta, 1)}\in G={\R}\times_{\Z}\t{{\rm SL}_2(\R)}$ by \cite[page
663]{BG1} and \cite[page 537]{BG2}, using a cohomological-dimension argument and the definition in paragraph 3.2. Suppose $\rho(s_i) = \o{(z_i,
x_i)}$. Then $s_i^{a_i}h^{b_i}=1$ implies that $$\o{({a_i}z_i,
x^{a_i})} \o{({b_i} \zeta,1)}=\o{({a_i}z_i+{b_i}\zeta, x^{a_i})}=1$$
Then there is an $n_i\in{\Z}$ such that (see Remark \ref{SL})
\begin{eqnarray}
{a_i}z_i+{b_i}\zeta=n_i \text{ in}\, \R \text{ and }
x_i \text{ is conjugate in $\t{{\rm SL}_2(\R)}$ to } {\rm sh}\left(-\frac{n_i}{a_i}\right)
\end{eqnarray}
Since $[\alpha_1,\beta_1]...[\alpha_g,\beta_g]=s_1....s_p$ and since the
product of commutators in ${\R}\times_{\Z}\t{{\rm SL}_2(\R)}$ must
lie  in $\t{{\rm SL}_2(\R)}$ this implies that
$$\o{(z_1+...+z_p, x_1...x_p)}=\o{\left(0,
\prod_{j=1}^g[\rho(\alpha_j),\rho(\beta_j)]\right)}$$
Then there is an $n\in{\Z}$ such that
\begin{eqnarray}
z_1+...+z_p=n \text{ and } \prod_{j=1}^g[\rho(\alpha_j),\rho(\beta_j)]=x_1...x_p{\rm sh}(n)
\end{eqnarray}
Equalities  (4.6) and (4.5),  imply
 condition (4.1) in Proposition \ref{SvSm} using  Theorem \ref{Eisenbud-Hirsch-Neumann} and
its proof in \cite{EHN}.
By (4.5) and (4.6), we can calculate directly
\begin{eqnarray}
z_i=\frac{n_i}{a_i}-\frac{b_i}{a_i}\zeta,\,\,\,  \zeta= \frac 1e \left(\sum_{i=1}^p\frac{n_i}{a_i}-n\right)
\end{eqnarray}
Plugging (4.5),  (4.6) and (4.7) into $\rho(h)= (\zeta, 1)$ and $\rho(s_i)
= (z_i, x_i)$, we obtain (4.2) and (4.3) in Proposition
\ref{SvSm}.
Then the  "moreover" part of Proposition \ref{SvSm} also follows from Theorem
\ref{Eisenbud-Hirsch-Neumann}.

Let's now compute the volume of such a representation. Let $p_1: \t
N \to N$ be a covering from  a circle bundle $\t N$ over $\t F$ to
$N$ so that the fiber degree is 1.  Then we have $$\t e=e(\t
N)=(\text{deg} p_1) e$$ Let $\t t$ be the fiber of $\t N$ and $\t
\rho =\rho|\pi_1\t N$. Then $\left(\t{t}\right) ^{\t e}=\prod_{j=1}^{\t
g}[\t{\alpha_j},\t{\beta_j}]$ in $\pi_1\t N$, and therefore $\t \rho
(\left(\t{t}\right) ^{\t e})= (\o{\t e\zeta, 1})\in Z(G)\cap \t{{\rm
SL}_2(\R)}$, since the image of the fibre must  lie in the center and
the image of the product of commutators must lie in $\t{{\rm
SL}_2(\R)}$. Hence $\t e \zeta= \t n \in{\Z}$.

Let $p_2: \hat N\to \t N$ be the covering along the fiber direction
of degree $\t e$, and then $\hat e = e(\hat N)=1$. Then $\hat \rho =
\t \rho|$ sends actually $\pi_1\hat N$  into  $\t{{\rm SL}_2(\R)}$
and the fiber $\hat t$ of $\hat N$ is sent to ${\rm sh}(\t{n})$. And
finally there is a covering $p_*: \hat N\to N^*$ along the fiber
direction of degree $\t n$, and where $N^*$ is a circle bundle over
a hyperbolic surface $F$ with $e^*=e(N^*)=\t n$. It is apparent that
$\hat \rho$ descends to $ \rho^*: \pi_1 N^*\to \t{{\rm SL}_2(\R)}$
such that $\rho^*(h^*)=\rm sh(1)$, where $h^*$ denotes the
${\S}^1$-fiber of $N^*$. According to Theorem
\ref{Eisenbud-Hirsch-Neumann}, there is a $({\rm PSL}_2(\R),{\S}^1)$-horizontal foliation on $N^*$, and according to Proposition
\ref{Euler}, ${\rm vol}(N^*,\rho^*)=4\pi^2e^*=4\pi^2\t n$, and then
$${\rm vol}(\hat{N},\hat \rho)=4\pi^2\t n^2=4\pi^2\t e^2\zeta^2.$$
Note that $$\text{deg} p_1 \text{deg} p_2= \frac{\t e}  e\times \t e=
\frac {\t e^2} e$$ By those facts  we reach (4.4) as
below:
$${\rm vol}(N,\rho)=\frac{{\rm vol}(\hat{N},\hat \rho)}{\text{deg} p_1 \text{deg} p_2}=\frac{4\pi^2\t e^2\zeta^2}{\frac {\t e^2} e}= 4\pi^2 {e}{\zeta^2}=4\frac{\pi^2}{|e|}\left(\sum_{i=1}^p\frac{n_i}{a_i}-n\right)^2$$
\end{proof}

\begin{remark} Suppose in Proposition \ref{SvSm} that $n_i=a_ik_i+r_i$, where $0\le r_i<a_i$. If we choose $n=2-2g+\sum_ik_i$ and $n_i=(k_i+1)a_i-1$ then
the corresponding representation $\rho_0$ is faithful, discrete and reaches the maximal volume
giving rise to the well known formula
$${\rm vol}(N,\rho_0)=4\pi^2\frac{1}{|e(N)|}\chi_{O(N)}^2$$
\end{remark}

\section{A brief review on the Chern Simons Theory}

Throughout this section we refer to
\cite{CS} and \cite{KN}. In this part, all the objects we deal with
are smooth. Let $\pi\co P\to M$ denote a principal $G$-bundle over a
closed manifold $M$. Suppose that $G$ is a Lie group acting on the right on $P$ and denote by $R_g$ the right action
$$P\ni x\mapsto x.g\in P$$ where $g$ in an element of $G$.
 Denote by $\mathfrak{g}$ the Lie algebra of $G$. Let $VP$ be the vertical subbundle of $TP$.

\subsection{Differential forms taking values in a Lie algebra}
  We denote by $\Omega^k(P; \mathfrak{g})$ the set of differential $k$-forms taking values in $\mathfrak{g}$.  We define the exterior product of $\omega_1\in\Omega^k(P; \mathfrak{g})$ by   $\omega_2\in\Omega^l(P; \mathfrak{g})$ as an element $\omega_1\wedge\omega_2$ of $\Omega^{k+l}(P; \mathfrak{g}\otimes\mathfrak{g})$ by setting
$$\omega_1\wedge\omega_2(X_1,...,X_{k+l})=$$ $$\frac{1}{(k+l)!}\sum_{\sigma\in\mathfrak{S}_{k+l}}{\rm sign}(\sigma)\omega_1(X_{\sigma(1)},...,X_{\sigma(k)})\otimes\omega_2(X_{\sigma(k+1)},...,X_{\sigma(l)})$$
The Lie bracket $[ ., . ]$ in $\mathfrak{g}$ induces a map $\Omega^{k}(P; \mathfrak{g}\otimes\mathfrak{g})\to\Omega^{k}(P; \mathfrak{g})$ and we denote by $[\omega_1,\omega_2]$ the image of $\omega_1\wedge\omega_2$ under this map. Explicitely we get:
$$[\omega_1,\omega_2](X_1,...,X_{k+l})=$$ $$\frac{1}{(k+l)!}\sum_{\sigma\in\mathfrak{S}_{k+l}}{\rm sign}(\sigma)[\omega_1(X_{\sigma(1)},...,X_{\sigma(k)}),\omega_2(X_{\sigma(k+1)},...,X_{\sigma(l)})]$$
The differrential $d\co\Omega^k(P; \mathfrak{g})\to\Omega^{k+1}(P; \mathfrak{g})$ is defined by the Cartan formula
$$d\omega(X_1,...,X_{k+1})=\frac{1}{k+1}\sum_{i=1}^{k+1}X_i.\omega(X_1,...,\hat{X_i},...,X_{k+1})+$$ $$\frac{1}{k+1}\sum_{i<j}(-1)^{i+j}\omega([X_i,X_j],X_1,...,\hat{X_i},...,\hat{X_j},...,X_{p+1})$$
\subsection{Connections on principal bundles}
The derivative at the identity $1$ of $G$ of the map $$G\ni g\mapsto x.g\in P$$ induces an isomorphism $\nu_x\co\mathfrak{g}\to V_xP\subset T_xP$ and we get the exact sequence
$$0\to\mathfrak{g}\stackrel{\nu_x}{\to} T_xP\stackrel{d\pi_x}{\to}T_{\pi(x)}M\to 0$$
A horizontal subbundle $HP$ of $TP$ is a smooth distribution such that $T_xP=V_xP\oplus H_xP$ for any $x\in P$ that is $G$ equivariant: $H_{x.g}=dR_g(x)H_x$. This is given equivalently by the kernel  of an element $\omega\in\Omega^1(P;\mathfrak{g})$ such that for any $x\in P$

(1) $\omega_x\circ\nu_x={\rm Id}_{\mathfrak{g}}$ and

(2) $R_g^{\ast}\omega={\rm Ad}_{g^{-1}}(\omega)$.

An element of $\Omega^1(P;\mathfrak{g})$ satisfying (1) and (2) is
termed  \emph{a  connection of} $P$. Denote by $\c{A}(P) $ the space
of all conections on $P$. This space is naturally acted by the
\emph{gauge group} denoted by  $\c{G}_P$  consisting of the
$G$-equivariant  bundle automorphisms of $P$.

The basic example is the group $G$ itself, viewed as a trivial
bundle over a point or more generally the trivialized bundle
$M\times G$ with the so-called Maurer-Cartan connection
$\omega_{{\rm M.C.}}=d(L_{g^{-1}}\circ\pi_2)$, where $L_g$ denotes
the left translation in $G$ and $\pi_2$  the projection of $P$ onto
$G$. This connection satisfies the Maurer Cartan equation, namely
$$d\omega_{{\rm M.C.}}=-\frac{1}{2}[\omega_{{\rm M.C.}},\omega_{{\rm
M.C.}}]$$

Let's make a concrete computation for $G$. Let
$X_1,...,X_n$ be a basis of $\mathfrak{g}$.   Since $\mathfrak{g}$ can be thought of as the space of left invariant vector fields in $G$, its dual $\mathfrak{g}^*$ is the space of left invariant differential 1-forms on $G$. Let
$\theta^1,...,\theta^n$ denote the dual basis of $\mathfrak{g}^*$. Then
$$\omega_{{\rm M.C.}}=\theta^1\otimes X_1+...+\theta^n\otimes X_n$$ Let us write
the constants structure of $\mathfrak{g}$ which are given by the
formula $$[X_j,X_k]=\sum_ic^i_{jk}X_i$$ Thus by  the Maurer-Cartan
equation we get the equalities
\begin{eqnarray}
d\theta^i=-\frac{1}{2}\sum_{j,k}c^i_{j,k}\theta^j\wedge\theta^k
\end{eqnarray}
In general,  for a given connection $\omega$ in a bundle $P$, the element
 \begin{eqnarray}
 F^{\omega}=d\omega+\frac{1}{2}[\omega,\omega]
 \end{eqnarray}
 is the \emph{curvature} of $\omega$ lying in $\Omega^2(P;\mathfrak{g})$ and it measures the integrability of the corresponding horizontal distribution.  When $F^{\omega}=0$ we say that the connection is flat. Denote by  $\c{FA}(P)$ the subset of $\c{A}(P)$ which consists of flat connections on $P$. This subspace is preserved by the gauge group action.

We recall the following basic fact that will be used very often in
this paper. To each flat connection $\omega$ one can associate a
representation $\rho\co\pi_1M\to G$ by lifting the loops of $M$ in
the leaves of the horizontal foliation given by integrating the
distribution $\ker\omega$.


On the other hand $\omega$ can be recovered from $\rho$ by the
following construction. The fundamental group of $M$ acts on the
product $\t{M}\times G$ by the formula
$[\sigma].([\gamma],g)=([\sigma.\gamma],\rho([\sigma]^{-1}).g)$ and
the quotient $\t{M}\times_{\rho} G$ under this $\pi_1M$-action is
isomorphic to $P$ and the push forward of the vertical distribution
of $\t{M}\times G$ in $\t{M}\times_{\rho} G$ corresponds to $\omega$
in $P$. We get a natural map
$$I_P\co\c{B}(P)=\c{FA}(P)/\c{G}_P\hookrightarrow\c{R}(\pi_1M,G)/{\rm conjugation}$$
where $\c{R}(\pi_1M,G)$ is the set of representations of $\pi_1M$
into $G$ acted by the conjugation in $G$. Notice that this map is usually non-surjective.

\subsection{The Chern Simons classes}
Given a  Lie group $G$, a polynomial of degree $l$ is a symmetric
linear map $f\co\otimes^l\mathfrak{g}\to {\mathbb K}$, where
${\mathbb K}$ denotes either the real or  the complex numbers field.
The group $G$ acts on $\mathfrak{g}$ by ${\rm Ad}$ and the
polynomials invariant under this action are called the
\emph{invariant polynomials of degree} $l$ and are denoted by
$I^l(G)$ with the convention $I^0(G)={\mathbb K}$. Denote $I(G)$ the
sum $\oplus_{l\in{\N}} I^l(G)$.

The \emph{Chern Weil theory} gives a correspondence $W_P$ from
$I^l(G)$ to $H^{2l}(M;{\mathbb K})$ constructed in the following
way. Choose a connection $\omega$   in $P$ then for any $l\geq 1$ a
polynomial $f\in I^l(G)$ gives rise to a $2l$-form
${f}(\wedge^lF^{\omega})$ in $P$. It follows from the Chern Weil
Theory that ${f}(\wedge^lF^{\omega})$ is closed and  is the
pull-back of a unique form on $M$ under $\pi\co P\to M$ denoted by
$\pi_{\ast}{f}(\wedge^lF^{\omega})$. Then $W_P(f)$ is by definition
the class of $\pi_{\ast}{f}(\wedge^lF^{\omega})$ in $H^{2l}(M)$. The
Chern Weil Theorem claims that $W_P(f)$ does not depend on the
chosen connection $\omega$ and that $W_P$ is actually a
homomorphism.

Let $EG$ denote the \emph{universal principal} $G$-bundle and denote by $BG$ the classifying space of $G$. This means that any principal $G$-bundle $P\to M$ admits a bundle homomorphism $\xi\co P\to EG$ descending to the classifying map, still denoted $\xi\co M\to BG$, that is unique up to homotopy. There exists the universal Chern Weil homomorphism $\t{W}\co I^l(G)\to H^{2l}(BG)$ such that $\xi^{\ast}\t{W}(f)=W_P(f)$.

The Chern Simons invariants were derived from this construction by Chern and Simons who
observed that ${f}(\wedge^lF^{\omega})$, for $l\geq 1$, is actually exact in $P$
and a primitive is given explicitly in \cite{CS} by
\begin{eqnarray}
Tf(\omega)=l\int_0^1f(\omega\wedge(\wedge^{l-1}F^t))dt
\end{eqnarray}
where $F^t=tF^{\omega}+\frac{1}{2}(t^2-t)[\omega,\omega]$. The form
$Tf(\omega)$ is closed when $M$ is of dimension $2l-1$. For instance
when $l=2$ and $M$ is a $3$-manifold, plugging $F^t$ and (5.2) into
(5.3) we get a closed $3$-form on $P$, namely
\begin{eqnarray}
Tf(\omega)=f(F^{\omega}\wedge\omega)-\frac{1}{6}f(\omega\wedge[\omega,\omega])=f(d{\omega}\wedge\omega)+\frac{1}{3}f(\omega\wedge[\omega,\omega])
\end{eqnarray}
Considering $G$ as a principal bundle over the point this yields to $$Tf(\omega_{{\rm M.C.}})=-\frac{1}{6}f(\omega_{{\rm M.C.}}\wedge[\omega_{{\rm M.C.}},\omega_{{\rm M.C.}}])$$ The $(2l-1)$-form $Tf(\omega_{{\rm M.C.}})$ is closed, bi-invariant and defines a class in $H^{2l-1}(G;{\R})$.
Let's denote by
$$I_0(G)=\{f\in I(G), Tf(\omega_{{\rm M.C.}})\in H^{2l-1}(G;{\Z})\}$$
 The elements of $I_0(G)$ are termed \emph{integral polynomials}. If $f\in I_0(G)$ then there is a well defined functional

\begin{eqnarray}
 \mathfrak{cs}^*_{M}\co\c{A}_{M\times G}\to{\mathbb K}/{\Z}
 \end{eqnarray}
 defined as follows: since $P=M\times G$  is a trivial(ized)  we can consider, for any section $\delta$, the Chern Simons invariant
 \begin{eqnarray}
 \mathfrak{cs}_M(\omega,\delta)=\int_M\delta^{\ast}Tf(\omega)
\end{eqnarray}
Since $f$ is an integral polynomial, the element $\mathfrak{cs}_M(\omega,\delta)$ is well defined modulo ${\Z}$ when the section changes. Then define $\mathfrak{cs}^{\ast}_M(\omega)$ to be   the class of $\mathfrak{cs}_M(\omega,\delta)$ in ${\mathbb K}/{\Z}$.

The fundamental classical examples  are $G={\rm SU}(2;{\C})$ and
$G={\rm SO}(3;{\R})$.

The Chern Simons classes for the group ${\rm SU}(2;{\C})$ are based on the second Chern class $f=C_2\in I_0^2({\rm SL}(2;{\C}))$.
We recall that the Chern classes,  denoted by $C_1$, $C_2$ for ${\rm SU}(2;{\C})$, are the complex valued invariant polynomials
such that $${\det}\left(\lambda. I_2-\frac{1}{2i\pi}A\right)=\lambda^2+C_1(A)\lambda+C_2(A\otimes A)$$  when $A\in\mathfrak{sl}_2(\C)$. Thus after developing this equality we get
$$C_2(A\otimes A)=\frac{1}{8\pi^2}{\rm tr}(A^2)$$ so that we get the usual
formula (using (5.4))
\begin{eqnarray}
TC_2(\omega)&=\frac{1}{8\pi^2}{\rm
Tr}\left(F^{\omega}\wedge\omega-\frac{1}{6}\omega\wedge[\omega,\omega]\right)
\end{eqnarray}
\begin{eqnarray*}
=\frac{1}{8\pi^2}{\rm
Tr}\left(d{\omega}\wedge\omega+\frac{1}{3}\omega\wedge[\omega,\omega]\right)
\end{eqnarray*}

The Chern-Simons classes of the special orthogonal group $G={\rm
SO}(3;{\R})$ are based on the first Pontrjagin class $f=P_1\in
I_0^2({\rm SO}(3;{\R}))$ that is a   the real valued invariant
polynomial such that $${\det}\left(\lambda.
I_3-\frac{1}{2\pi}A\right)=\lambda^3+P_1(A\otimes A)\lambda$$  when
$A\in\mathfrak{so}_3(\R)$.
Thus after developing this equality we get

$$  P_1(A\otimes A)=-\frac{1}{8\pi^2}{\rm tr}(A^2)$$

  \begin{example}
  When $M$ is an oriented Riemaniann closed $n$-manifold one can consider its associated ${\rm SO}(n;{\R})$-bundle ${\rm SO}(M)$  which consists of the  positive orthonormal unit frames endowed with the \emph{Levi Civita connection}.   When $M$ is of dimension $3$ it is well known that its is parallelizable so that there exist sections $\delta$ of  ${\rm SO}(M)\to M$. Therefore one can consider the Chern-Simons invariant of the Levi Civita connection on $M$ that will be denoted by $\mathfrak{cs}_{\rm L.C}(M,\delta)$.
  \end{example}

A natural question arises in the following situation. There is an
epimorphism $\pi_2\co{\rm SU}(2;{\C})\to{\rm SO}(3;{\R})$ that is
the $2$-fold universal covering. Thus any connection $\omega$ on the
trivialized ${\rm SU}(2;{\C})$-bundle over $M$ induces a connection
$\omega'$ on the corresponding   $ {\rm SO}(3;{\R})$-bundle over
$M$. How can we compute $TP_1(\omega')$ form $TC_2(\omega)$? The
answer is given in \cite[pp 543, end of Section 3]{KK} by recalling
that $\pi_2$ induces a homomorphism between the corresponding
classifying spaces  $$\pi_2^{\ast}\co H^4(B{\rm SO}(3;{\R}))\to
H^4(B{\rm SU}(2;{\C}))$$ such that
$$\pi_2^{\ast}\t{W}(P_1)=-4\t{W}(C_2)$$ Thus using the definition
and the Chern Weil universal  homomorphism we get the equality
\begin{eqnarray}
\mathfrak{cs}_{M}(\omega',\delta')=-4\mathfrak{cs}_{M}(\omega,\delta)
\end{eqnarray}

where $\delta$ is a fixed section in the ${\rm SU}(2;{\C})$-bundle
over $M$ and $\delta'$ is the corresponding section in the $ {\rm
SO}(3;{\R})$-bundle over $M$. On the other hand since $G={\rm
SO}(3;{\R})$, resp. ${\rm SU}(2;{\C})$, are the maximal compact
subgroup of ${\rm PSL}(2;{\C})$, resp. ${\rm SL}(2;{\C})$, whose
quotients ${\rm PSL}(2;{\C})/{\rm SO}(3;{\R})$, resp. ${\rm
SL}(2;{\C})/{\rm SU}(2;{\C})$ are contractible then it follows from
\cite{Ho}[Chapter 15, Theorem 3.1] and \cite{Du1}[Proposition 7.2,
p. 98] that  the natural inclusion gives rise to isomorphisms
$H^{\ast}(B {\rm PSL}(2;{\C}))\to H^{\ast}(B {\rm SO}(3;{\R}))$ and
$H^{\ast}(B {\rm SL}(2;{\C}))\to H^{\ast}(B {\rm SU}(2;{\C}))$. We
have the following commutative diagram
$$\xymatrix{
H^*(B{\rm PSL}(2;{\C})) \ar[r]^{\simeq} \ar[d] & H^*(B {\rm SO}(3;{\R})) \ar[d]\\
H^*(B{\rm SL}(2;{\C})) \ar[r]^{\simeq} & H^*(B {\rm SU}(2;{\C})) }$$
Hence we also get (fixing a trivialization, using (5.6), (5.7),
(5.8))
\begin{eqnarray}
 \mathfrak{cs}_{M}(\omega',\delta')&=-4\mathfrak{cs}_{M}(\omega,\delta)
 \end{eqnarray}
\begin{eqnarray*} =
-\frac{1}{2\pi^2}\int_M\delta^{\ast}{\rm
Tr}\left(F^{\omega}\wedge\omega-\frac{1}{6}\omega\wedge[\omega,\omega]\right)
\end{eqnarray*}
 where $\delta$ is a fixed section in the ${\rm
SL}(2;{\C})$-bundle over $M$ and $\delta'$ is the corresponding
section in the $ {\rm PSL}(2;{\C})$-bundle over $M$.

\subsection{Volume and Chern Simons classes in Seifert geometry}
In this section we check Proposition \ref{vol} keeping the same notation as in the introduction.
The  proof is inspired from \cite[p. 532]{BG2} and we  we will
follow faithfully their presentation. If $G={\rm Isom}_e(\t{{\rm
SL}_2({\R})})$ then the matrices $$X=\begin{pmatrix}
1&0 \\
0&-1
\end{pmatrix} , Y=\begin{pmatrix}
0&0 \\
1&0
\end{pmatrix} , Z=\begin{pmatrix}
0&1 \\
0&0
\end{pmatrix}$$ together with the generator $T$ of   ${\R}$ form a basis of the Lie algebra $\mathfrak{g}$ of $G={\rm Isom}_e(\t{{\rm SL}_2({\R})})$.
Setting $W=Z-Y-T$ we get a new basis $\{X,Y,Z,W\}$ of $\mathfrak{g}$
with commutators relations
\begin{eqnarray}
[X,Y]=-2Y, [X,Z]=2Z,
\end{eqnarray}
\begin{eqnarray*}
[Y,Z]=[Y,W]=[Z,W]=-X, [X,W]=2Y+2Z
\end{eqnarray*}
 which determine the
coefficients in the Maurer Cartan equations. Denote by
$\varphi_X,\varphi_Y,\varphi_Z,\varphi_W$ the dual basis of
$\mathfrak{g}^{\ast}$. The Maurer Cartan form of $G$ is given by
$$\omega_{\rm M.C.}=\varphi_X\otimes X+\varphi_Y\otimes
Y+\varphi_Z\otimes Z+\varphi_W\otimes W$$ Denote by $A$ a flat
connection on $M\times_{\rho}G$. By section 5.2,  if $\t{M}$ denotes
the universal covering and if  $q\co \t{M}\times G\to G$ denotes the
projection then  $A$ corresponds to the form
$\o{q^{\ast}(\omega_{\rm M.C.})}$, where $-\co\t{M}\times
G\to\t{M}\times_{\rho} G$ denotes the push-forward which makes sense
since $q^{\ast}(\omega_{\rm M.C.})$ is $\pi_1M$-invariant.   The
Chern Simons class of the flat connection  $A$ is $T{\bf
R}(A)=\o{q^{\ast}T{\bf R}(\omega_{\rm M.C.})}$. Using  equations
(5.1) and (5.10), we calculate
\begin{eqnarray}
d\varphi_X= & \varphi_Y\wedge\varphi_Z+\varphi_Y\wedge\varphi_W+\varphi_Z\wedge\varphi_W\\
d\varphi_Y= & 2\varphi_X\wedge\varphi_Y-2\varphi_X\wedge\varphi_W\\
d\varphi_Z= & -2\varphi_X\wedge\varphi_Z-2\varphi_X\wedge\varphi_W\\
d\varphi_W= & 0
\end{eqnarray}
Notice that those equations also imply that
$2(\varphi_X\wedge\varphi_Y+\varphi_X\wedge\varphi_Z)=d(\varphi_Y-\varphi_Z)$
and therefore
$$T{\bf R}(\omega_{\rm M.C.})=\frac{2}{3}\varphi_X\wedge\varphi_Y\wedge\varphi_Z+\frac{1}{3}  d(\varphi_Y\wedge\varphi_W-\varphi_Z\wedge\varphi_W)$$
The end of the proof follows from the  commutativity of the diagram
below and from the Stokes formula, since
$\varphi_X\wedge\varphi_Y\wedge\varphi_Z$ represents the volume form
on $X=\t{{\rm SL}_2({\R})}$.
$$\xymatrix{
G \ar[r] & X \\
 \t{M}\times  G\ar[u]^{q_G} \ar[d]_{-} \ar[r]^{\t{\pi}} & \t{M}\times X \ar[u]_{q_X}  \ar[d]^{-}  \\
  M\times_\rho G \ar[r]^{\pi}  & M\times_\rho X \\
M \ar[u]^{\delta} \ar[ur]_{s} & }$$ This completes the proof of
the proposition.

\subsection{Volume  and Chern Simons classes in Hyperbolic geometry}
We now check Proposition \ref{volh}.  The following construction is largely inspired from  \cite[p.
553-556]{KK}, using a formula established by Yoshida in \cite{Yo}.

Denote by $p\co{\rm PSL}(2;{\C})\simeq{\rm Iso}_+{\Hi}^3\to{\Hi}^3$ the the natural projection.
For short denote ${\rm PSL}(2;{\C})$ by $G$.
For each representation $\rho\co\pi_1M\to G$ admitting a lift into ${\rm SL}(2;{\C})$, we have the (trivial)
principal bundle $M\times _\rho G$ and the associated bundle
$M\times _\rho {\Hi}^3$. Denote by $A$ the flat connection over $M$
corresponding to $\rho$ and $\omega_{{\Hi}^3}$ the $G$-invariant volume form on
${\Hi}^3$ corresponding to  the hyperbolic metric.

The matrices $X=\begin{pmatrix}
1&0 \\
0&-1
\end{pmatrix}$, $Y=\begin{pmatrix}
0&0 \\
1&0
\end{pmatrix}$, $Z=\begin{pmatrix}
0&1 \\
0&0
\end{pmatrix}$ form a basis of the Lie algebra $\mathfrak{sl}(2;{\C})$ with commutators relations $$[X,Y]=-2Y, [X,Z]=2Z, [Y,Z]=-X$$ Denote by $\varphi_X,\varphi_Y,\varphi_Z$ the dual basis of $\mathfrak{sl}^{\ast}(2;{\C})$. The Maurer Cartan form of $G$ is  $$\omega_{\rm M.C.}=\varphi_X\otimes X+\varphi_Y\otimes Y+\varphi_Z\otimes Z$$
and
$$TP_1(\omega_{\rm M.C.})=\frac{1}{\pi^2}\varphi_X\wedge\varphi_Y\wedge\varphi_Z$$
By the formula of Yoshida in \cite{Yo} we know that
$$iTP_1(\omega_{\rm M.C.})=\frac{1}{\pi^2}p^*\omega_{{\Hi}^3}+i\mathfrak{cs}_{\rm L.C.}({\Hi}^3)+d\gamma$$
where $p^*\omega_{{\Hi}^3}$ is the pull-back of $\omega_{{\Hi}^3}$
under the projection $p: {\rm PSL}(2;{\C})\to{\Hi}^3$,
$\mathfrak{cs}_{\rm L.C.}({\Hi}^3)$ is the Chern Simons 3-form of
the Levi Civita connection over ${\Hi}^3$ (see example 5.1) endowed with the
hyperbolic metric   and $d\gamma$ is an exact real form. Let's consider
the following commutative diagram
$$\xymatrix{
G \ar[r]^{{p}} & {\Hi}^3 \\
 \t{M}\times  G\ar[u]^{q_G} \ar[d]_{-} \ar[r]^{{p}} & \t{M}\times {\Hi}^3 \ar[u]_{q_{{\Hi}^3}}  \ar[d]^{-}  \\
  M\times_\rho G \ar[r]^{p}  & M\times_\rho {\Hi}^3 \\
M \ar[u]^{\delta} \ar[ur]_{s} & }$$
Notice that the sections in the bottom triangle are obtained as follows. Since $M$ is a $3$-manifold then it follows from the obstruction theory that any principal bundle with simply connected group is trivial. Since $\rho\co\pi_1M\to G$ admits a lift into ${\rm SL}(2;{\C})$ then  $M\times_\rho G$ is trivial.  So denote by $\delta$ a section of  $M\times_\rho G\to M$. It induces, by $p\circ\delta=s$, a section of  $M\times_\rho {\Hi}^3\to M$.

Since all the maps are clear from the context, in the sequel, we will drop the index in the projections $q_G$ and $q_{{\Hi}^3}$ and we denote them just by $q$.
 Now the 3-form $\omega_{{\Hi}^3}$ induces a 3-form
$\o{q^*\omega_{{\Hi}^3}}$ on $M\times_\rho {\Hi}^3$ and

$$i\o{q^*TP_1(\omega_{\rm M.C.})}=\frac{1}{\pi^2}\o{q^*p^*\omega_{{\Hi}^3}}+i\o{q^*\mathfrak{cs}_{\rm L.C.}({\Hi}^3)}+\o{q^*d\gamma}$$
in $M\times_\rho G$, where the push-forward operation $\o{q^{\ast}(.)}$  indeed makes sense since $TP_1(\omega_{\rm M.C.}), p^*\omega_{{\Hi}^3}$ and $\mathfrak{cs}_{\rm L.C.}({\Hi}^3)$ are left invariant forms in $G$. Then $$i\mathfrak{cs}_M(A,
\delta)=\frac{1}{\pi^2}\int_M\delta^*\o{q^*p^*\omega_{{\Hi}^3}}+i\int_M\delta^*\o{q^*\mathfrak{cs}_{\rm L.C.}({\Hi}^3)}+\int_M\delta^*\o{q^*d\gamma}$$
Since
$\delta^*\o{q^*p^*\omega_{{\Hi}^3}}=\delta^*p^*\o{q^*\omega_{{\Hi}^3}}=s^*\o{q^*\omega_{{\Hi}^3}}$
and $\int_M\delta^*\o{q^*d\gamma}=0$ by the Stokes formula,  we have
$$i\mathfrak{cs}_M(A,
\delta)=\frac{1}{\pi^2}\int_Ms^*\o{q^*\omega_{{\Hi}^3}}+i\int_M\delta^*\o{q^*\mathfrak{cs}_{\rm L.C.}({\Hi}^3)}=\frac{1}{\pi^2}{\rm
vol}(M,\rho)+i\mathfrak{cs}(M_{\rho};\delta),$$
where we denote $\int_M\delta^*\o{q^*\mathfrak{cs}({\Hi}^3)}$ by
 $i\mathfrak{cs}(M_{\rho};\delta)$.
 We get eventually
 \begin{eqnarray}
\mathfrak{cs}_M(A,\delta)=\mathfrak{cs}(M_{\rho};\delta)-\frac{i}{\pi^2}{\rm vol}(M,\rho)
\end{eqnarray}

\subsection{Normal form near toral boundary of  3-Manifolds}
In this part we recall the machinery developed in \cite{KK}.
Let $M$ be a compact oriented $3$-manifold with connected and toral
boundary $T=\b M$ endowed with a  basis $s,h$ of $H_1(\b M;{\Z})$. Notice that we study the case of manifold with connected toral boundary only to simplify the notations, for all the results stated in this section extend naturally to compact $3$-manifolds with non-connected toral boundary.
Let $\rho\co\pi_1M\to G$ be a representation where $G$ is either
${\rm PSL}(2;{\C})$ or $\t{{\rm
SL}(2;{\R})}$. We consider the space of flat connections $\c{FA}({P})$ where $P$ is the trivialized bundle $M\times G$. For representations into  $\t{{\rm
SL}(2;{\R})}$ the corresponding principal bundles  are always trivial whereas the representations $\rho$ into  ${\rm PSL}(2;{\C})$ leading to a trivial bundle are precisely those who admit a lift $\o{\rho}$ into ${\rm SL}(2;{\C})$. Moreover if follows from \cite{KK} and \cite{Kh} that after a conjugation, the representation $\rho|\pi_1T$ can be put in \emph{normal form}, which either \emph{hyperbolic, elliptic or parabolic}. Since the parabolic form won't be used in the explicit way we only recall  the definitions of  those representations which are \emph{elliptic/hyperbolic} in the  ${\rm PSL}(2;{\C})$-case and \emph{elliptic} in the  $\t{{\rm
SL}(2;{\R})}$-case in the boundary of $M$.   By \cite{KK}, when $G={\rm PSL}(2;{\C})$, we
may assume, after conjugation, that there exist $\alpha,\beta\in{\C}$ such that
$$\rho(s)=\begin{pmatrix}
e^{2i\pi\alpha}&0 \\
0&e^{-2i\pi\alpha}
\end{pmatrix}\  {\rm and}\ \rho(h)=\begin{pmatrix}
e^{2i\pi\beta}&0 \\
0&e^{-2i\pi\beta}
\end{pmatrix}$$
When $G=\t{{\rm
SL}(2;{\R})}$ we may assume that  after  projecting into ${\rm PSL}(2;{\R})$ there exist $\alpha,\beta\in{\R}$ such that, up to conjugation,
$$s\mapsto\begin{pmatrix}
\cos(2\pi\alpha)&\sin(2\pi\alpha) \\
-\sin(2\pi\alpha)&\cos(2\pi\alpha)
\end{pmatrix}\ {\rm and}\ h\mapsto\begin{pmatrix}
\cos(2\pi\beta)&\sin(2\pi\beta) \\
-\sin(2\pi\beta)&\cos(2\pi\beta)
\end{pmatrix}$$ In either case, if  $A$ denotes a connection on $P$ corresponding to $\rho$ then after a gauge transformation $g$ the connection $g*A$ is in normal form: $$g*A|T\times[0,1]=(i\alpha dx+i\beta dy)\otimes X$$

We quote the  following result  stated in \cite[Lemma 3.3]{KK} with $G={\rm SL}(2;{\C})$ and in \cite[Theorem 4.2]{Kh} with $G=\t{{\rm SL}(2;{\R})}$, that will be used latter :
\begin{proposition}\label{KKh}
Let $A$ and $B$ denote two flat connections in normal form over an oriented $3$-manifold with toral boundary. If $A$ and $B$ are equal near the boundary and if they are gauge equivalent then

i)  $\mathfrak{cs}_M(A,\delta)=\mathfrak{cs}_M(B,\delta')$ when $G=\t{{\rm SL}(2;{\R})}$   and,

ii)  $\mathfrak{cs}_M(A,\delta)-\mathfrak{cs}_M(B,\delta')\in{\Z}$ when $G={\rm PSL}(2;{\C})$.
\end{proposition}
The second statement follows from  \cite[Lemma 3.3]{KK} using
identity (5.9).
\begin{remark}\label{auto} As a consequence of  Proposition \ref{KKh}, if   $A$ and $B$ are flat connections on a solid torus that are equal near the boundary then the associated representations are automatically conjugated so that the conclusion of the proposition applies.
\end{remark}
Since the complex function $z\mapsto e^{2i\pi z}$ identifies
${\C}/{\Z}$ with ${\C}^*$,    it will be convenient to use rather the
invariant (see also (5.5))
$$\mathfrak{cs}^*_M(A)=e^{2i\pi\mathfrak{cs}_M(A,\delta)}$$ when
$G={\rm PSL}(2;{\C})$ and $A$ is a connection in normal form, 
than $\mathfrak{cs}_M(A,\delta)$. Denote by $\c{ER}^0_M({\rm
PSL}(2;{\C}))$ the image, up to conjugation, of the
elliptic/hyperbolic representations of $\pi_1M$ into ${\rm
SL}(2;{\C})$ induced by the projection ${\rm SL}(2;{\C})\to{\rm
PSL}(2;{\C})$.

There is a natural map $t\co{\C}^2\to\c{ER}^0_{\b M}({\rm PSL}(2;{\C}))$ which sends $(\alpha,\beta)\in{\C}^2$ to the conjugation class of the  representation inducing the homomorphism ${\Z}\times{\Z}\to{\rm PSL}(2;{\C})$ defined by
$$s\mapsto\begin{pmatrix}
e^{2i\pi\alpha}&0 \\
0&e^{-2i\pi\alpha}
\end{pmatrix}\ \ {\rm and}\ \ h\mapsto\begin{pmatrix}
e^{2i\pi\beta}&0 \\
0&e^{-2i\pi\beta}
\end{pmatrix}$$ on the boundary.
Let $\rho$ be an element of $\c{ER}^0_M({\rm PSL}(2;{\C}))$. Then the map $\rho\mapsto(\alpha,\beta)$ is a lifting $L$ of the restriction map $r\co\c{ER}^0_M({\rm PSL}(2;{\C}))\to\c{ER}^0_{\b M}({\rm PSL}(2;{\C}))$ such that the following diagram commutes
$$\xymatrix{
 & {\C}^2 \ar[d]^t\\
\c{ER}^0_M({\rm PSL}(2;{\C}))  \ar[ur]^L \ar[r]^r & \c{ER}^0_{\b M}({\rm PSL}(2;{\C}))
}$$
To any $\rho$ in $\c{ER}^0_M({\rm PSL}(2;{\C}))$  we associate, after fixing a lift $L$ of $r$, the triple  $$(\alpha,\beta,\mathfrak{cs}^*_M(A))\in{\C}^2\times{\C}^{\ast}$$ where $A$ is the connection over $M$ corresponding to $L(\rho)$ in normal form $$A|T\times[0,1]=(i\alpha dx+i\beta dy)\otimes X$$ near the boundary. Using \cite[Theorem 2.5]{KK} with the group ${\rm PSL}(2;{\C})$ instead of ${\rm SU}(2;{\C})$ it turns out that if we replace  $A$ by $B$ with a lift $(\alpha+1/2,\beta,\mathfrak{cs}^*_M(B))$ or by $C$ with a lift $(\alpha,\beta +1/2,\mathfrak{cs}^*_M(C))$ then we get the following equalities:
\begin{eqnarray}
\mathfrak{cs}^*_M(B)=\mathfrak{cs}^*_M(A)e^{-4i\pi\beta} \ \ {\rm and} \ \ \mathfrak{cs}^*_M(C)=\mathfrak{cs}^*_M(A)e^{4i\pi\alpha}
\end{eqnarray}
whereas $t(\alpha,\beta)=t(\alpha+1/2,\beta)=t(\alpha,\beta +1/2)$.
We end this section by quoting the following result established in \cite[Theorem 2.7]{KK} for the group ${\rm SU}(2;{\C})$.
\begin{proposition}\label{path} Let $M$ be an oriented compact $3$-manifold with toral bondary $\b M=T$.
Let $\rho_t\co\pi_1M\to{\rm PSL}(2;{\C})$ be a path of elliptic/hyperbolic representations in $\c{ER}^0_M({\rm PSL}(2;{\C}))$. Choose a lifting $L(\rho_t)=(\alpha(t),\beta(t))$ of the restriction $r\co\c{ER}^0_M({\rm PSL}(2;{\C}))\to\c{ER}^0_{\b M}({\rm PSL}(2;{\C}))$ and denote by $A_t$ the corresponding path of flat connections.  Then
$$\mathfrak{cs}^*_M(A_1)=\mathfrak{cs}^*_M(A_0)\exp\left(-8i\pi\int_0^1(\alpha(t)\beta'(t)-\beta(t)\alpha'(t))dt\right)$$
\end{proposition}
As a corollary we quote the formula stated in \cite[end of p. 555]{KK}:
\begin{corollary}\label{adding} Let $V={\bf D}^2\times{\S}^1$ denote the solid torus and let $\rho\co\pi_1V\to{\rm PSL}(2;{\C})$ denotes a representation. If $L(\rho)=(1/2,\beta)$ with respect to the meridian-longitude basis on $\b V=\b{\bf D}^2\times{\S}^1$ then
$$\mathfrak{cs}_M^*(B)=\exp(-4i\pi\beta)$$
where $B$ is the flat connection over $V$ corresponding to the lifting $L(\rho)$.
\end{corollary}

\begin{proof}
We begin by computing $\mathfrak{cs}^*_M(A)$ where $A$ is a
connection in normal form  corresponding to the lift
$(0,\beta)\in{\C}^2$. To this purpose we consider the path of flat
connections $\omega_t$ given by $(0,t\beta)$. Since $\omega_0$ is
the trivial connections then $\mathfrak{cs}_M^*(\omega_0)=1$ and
Proposition \ref{path} implies $\mathfrak{cs}_M^*(\omega_1)=1$.
Applying formula (5.16) once gives rise to $(1/2,\beta)$ with the
corresponding connections whose ${cs}_M^*$ is equal to $e^{-4i\pi
\beta}$. This proves the formula.
\end{proof}

 \subsection{Additivity principle}

 Fix a closed oriented  3-manifold $M$ and denote by $[M]$ its
 orientation class. Let $T$ be a separating torus cutting  $M$ into
 $ M_1$ and $M_2$. Denote by $[M_1, \partial M_1]$ and $[M_2, \partial M_2]$
 the induced orientations classes so that
  the induced orientations on $\partial M_1$ and $\partial M_2$
 are opposite on $T$, and we have $[M]=[(M_1,\b
M_1)]+[(M_2,\b M_2)]$.

Fix a regular neighbourhood $W(T)=[0,1]\times T$ such that
$T=\{1/2\}\times T$, $M_1\cap W(T)=[0,1/2]\times T$ and  $M_2\cap
W(T)=[1/2,1]\times T$. Let $A$ denote a flat connection over $M$.
Applying the same arguments as in \cite{KK} we may assume that  $A|W(T)$
is in normal form. Then by linearity of the integration
$$\mathfrak{cs}^*_{M}(A)=\mathfrak{cs}^*_{M_1}(A|M_1)\mathfrak{cs}^*_{M_2}(A|M_2)$$
Denote by $V$ the solid torus ${\bf D}^2\times{\S}^1$ with meridian
$m$.
Denote by $c$ a slope in $T$ and for each $i=1,2$ we perform a Dehn
filling to $M_i$ identifying $c$ with $m$ and denote by
$\hat{M}_i=M_i\cup V$ the resulting closed oriented manifold.
Suppose both $A|M_1$ and $A|M_2$
smoothly extend to  flat connections over $\hat{M}_1$ and
$\hat{M}_2$, respectively denoted by $\hat{A}_1$ and $\hat{A}_2$. This is to say that for any representation $\rho$ corresponding to $A$ then $[c]\in\ker\rho$.  By the linearity we have

$$\mathfrak{cs}^*_{\hat{M}_1}(\hat{A}_1)\mathfrak{cs}^*_{\hat{M}_2}(\hat{A}_2)=\mathfrak{cs}^*_{M_1}(A|M_1)\mathfrak{cs}^*_{V}(\hat{A}_1|V)\mathfrak{cs}^*_{M_2}(A|M_2)\mathfrak{cs}^*_{V}(\hat{A}_2|V).$$

Since the extensions from $M_i$, $i=1,2$,  to $V$, based on the
normal form on $[0,1]\times T$, are the same on the $T$ direction but
opposite on the $[0,1]$ direction, then using Proposition \ref{KKh} and  Remark \ref{auto}
$$\mathfrak{cs}^*_{V}(\hat{A}_1|V)\mathfrak{cs}^*_{V}(\hat{A}_2|V)=1$$
Then applying equalities (1.2) of Proposition \ref{vol} and  (1.3) in Proposition \ref{volh}  to the former
equality we get
\begin{eqnarray}
{\rm
vol}_G(M,\rho)={\rm vol}_G({M}_1(c),\hat{\rho}_1)+ {\rm
vol}_G({M}_2(c),\hat{\rho}_2)
\end{eqnarray}
 for $G=\t{{\rm SL}_2(\R)}$ or ${\rm PSL}(2;{\C})$,  and where $\hat{\rho}_i$ is
the extension of $\rho|\pi_1M_i$ to $\pi_1\hat{M_i}$, $i=1,2$.

\section{Manifolds with virtually positive hyperbolic volumes}

This section is devoted  to the  proof of Proposition \ref{hyper-virt-nonzero}.
To this purpose we
still need to make some technical statements. Consider a  compact
oriented $3$-manifold $Q$ with connected, toral boundary.
Denote by $i\co \b Q\to Q$ the natural inclusion. Applying the homology
exact sequence to the pair $(Q,\b Q)$ with real coefficients, we get
immediately
$${\rm Rank}\left(H_1\left(\b Q;{\R}\right)\stackrel{i_{\sharp}}{\to} H_1\left(Q;{\R}\right)\right)=1$$
Therefore   we may choose a \emph{meridian-longitude} basis
$(\mu,\lambda)$ of $H_1(\b Q;{\Z})$ such that $i_{\sharp}(\lambda)$ has
infinite order  whereas $i_{\sharp}(\mu)$ is a torsion element in
$H_1(Q;{\Z})$.
When $Q$ has non-connected toral boundary we write $\b Q=T_1\cup...\cup T_r$  and we fix a basis $(\lambda_i,\mu_i)$ on each component of $T_i$ as in Theorem \ref{surgered}.

In the sequel we will use the following covering result which was
first  proved in \cite{He} for the hyperbolic case and then extended 
 in \cite{Lu}.
\begin{lemma}\label{cover}
Let $Q$ be a compact, oriented and irreducible $3$-manifold with
toral boundary. Then   there exists a prime number $q_0$, depending
only on $Q$, such that for any  prime number $q\geq q_0$ there
exists a finite covering $p\co\t{Q}\to Q$ inducing the $q\times
q$-characteristic covering over  $\b Q$ such that each Seifert piece
of $\t{Q}$ is a product of a  surface with positive genus by the
circle.
\end{lemma}
We next state
\begin{lemma}\label{hvp}
 Let $Q$ be a compact oriented $3$-manifold with toral boundary  whose interior admits a complete (finite volume) hyperbolic metric. Then there exists a prime number $p_0$, depending only on $Q$, such that for any family of slopes $(m_1,...,m_r)$ in $\b Q$ with $m_i\subset T_i$ for $i=1,...,r$ and for any   prime number $q\geq p_0$ there exists a  finite covering   $p\co\t{Q}\to Q$ inducing the $q\times q$-characteristic covering over  $\b Q$ and a representation $\rho\co\pi_1\t{Q}(p^{-1}(m_1\cup...\cup m_r))\to{\rm PSL}(2;{\C})$ of positive volume, where $\t{Q}(p^{-1}(m_1\cup...\cup m_r))$ denotes the closed surgered manifold obtained from $\t{Q}$ after performing a Dehn filling on each component $U_i^j$ of $p^{-1}(T_i)$ identifying the meridian of a solid torus with a component of $p^{-1}(m_i)\cap U^j_i$ for any $i=1,...,r$.
\end{lemma}
\begin{proof}
Choose a prime number $p_0$ such that $p_0>\max\{C,q_0\}$, where $C$
is given in Theorem \ref{surgered}  and $q_0$ is the prime number
given in  Lemma \ref{cover},  applied to $Q$. Thus given a slope
$m_i=a_i\lambda_i+b_i\mu_i$ in $T_i$, with $(a_i,b_i)$ co-prime,
then for any prime number $q\geq p_0$ we have $\|(qa_i,qb_i)\|_2> C$
for $i=1,...,r$. Therefore we can apply Theorem \ref{surgered}
implying that $Q((qa_1,qb_1),...,(qa_r,qb_r))$ is a hyperbolic
orbifold. In particular there exists a representation
$\rho\co\pi_1Q((qa_1,qb_1),...,(qa_r,qb_r))\to{\rm PSL}(2;{\C})$
with positive volume and such that for each $i=1,...,r$ the element
$\rho(m_i)$ is conjugated to  $\begin{pmatrix}
e^{i\pi/q}&0 \\
0&e^{-i\pi/q}
\end{pmatrix}$ and $\rho(l_i)$ is conjugated to
$\begin{pmatrix}x_i&0 \\
0&x_i^{-1}
\end{pmatrix}$ for some $x_i\in{\C}^{\ast}$, where $l_i$ is the slope $c_i\lambda_i+d_i\mu_i$ in $T_i$ with $a_id_i-b_ic_i=1$. On the other hand, Lemma \ref{cover} applies to $Q$ for any prime number $q\geq p_0$ which gives rise to a finite covering $p\co\t{Q}\to Q$. Since it induces the $q\times q$-characterstic covering on the boundary then it induces an orbifold finite covering
$\hat{p}\co\t{Q}(p^{-1}(m_1\cup...\cup m_r))\to Q((qa_1,qb_1),...,(qa_r,qb_r))$. Accordingly the composition $\rho\circ\hat{p}_{\ast}$ is a representation of the (non-singular) manifold $\t{Q}(p^{-1}(m_1\cup...\cup m_r))$ whose volume is positive. This completes the proof of the lemma.
\end{proof}

\begin{proof}[Proof of Proposition \ref{hyper-virt-nonzero}]

Suppose that $N$ has a hyperbolic piece  $Q$ such that
each non-separating component of $\partial Q$ in $N$ is shared by a
Seifert piece. Let $N=Q\cup(Q_1\cup...\cup Q_l)$, where each $Q_i$
is a component of $N\setminus Q$.  For each component $Q_i$ with
connected boundary we fix  its meridian-longitude basis
$(\mu_i,\lambda_i)$; and for each component $Q_j$ with non-connected
boundary $T_j^1,...,T_j^{l_j}$, by the condition posed on $Q$  we denote
by $h^l_j$, $l=1,...,l_j$ the regular fiber of the Seifert piece
adacent to $Q$ represented in $T^l_j$.

For each component $Q_i$ denote by $q_i$ the prime number  such that
for any prime number $q\geq q_i$ there exists a $q\times
q$-characteristic finite covering map $\t{Q}_i\to Q_i$ satisfying
the conclusion of Lemma \ref{cover}. Notice that when $\b Q_i$ is
connected, for each component ${T}_i^k\subset\b\t{Q}_i$,
$k=1,...,r_i$,  over $\b Q_i$ then each component ${\mu}^k_i$ of
$p^{-1}(\mu_i)\cap{T_i^k}$ and ${\lambda}^k_i$ of
$p^{-1}(\lambda_i)\cap{T_i^k}$ is a basis of $H_1({T_i^k};{\Z})$.
When $\b Q_j$ is non-connected then we fix a trivialization of the
Seifert pieces adjacent to $\b\t{Q}_j$ providing a section-fiber
basis of $H_1(T_j^l;{\Z})$ for each component of $\b\t{Q}_j$.

We apply Lemma \ref{hvp} to $Q$ with the family of slopes
$\{\mu_i\}_i$ and $\{h^l_j\}_{l,j}$   and we denote by $q$ a prime
number such that $q>\max\{p_0,q_1,...,q_l\}$ with the corresponding
covering $\t{Q}\to Q$.     We denote by $p\co{M}\to N$ a finite
covering such that each component of $p^{-1}(Q_i)$, resp.
$p^{-1}(Q)$  is homeomorphic to $\t{Q}_i$ for $i=1,...,l$, resp. to
$\t{Q}$. Such a covering can be constructed following the arguments
of  \cite{Lu}.

For each component $\t{Q}$ of $p^{-1}({Q})$, let $\hat
Q=\t{Q}\left(\cup_ip^{-1}\left(\mu_i\right)\cup_{l,j}p^{-1}\left(h^l_j\right)\right)$.
By Lemma \ref{hvp}, there exists a representation
$$\rho\co\pi_1\hat
Q\to{\rm PSL}(2;{\C})$$ such that $${\rm vol}\left(\hat
Q,\rho\right)>0$$ and satisfying the following conditions:

(1) when $\b Q_i$ is connected  $\rho({\mu}^k_i)$ is trivial and
$\rho({\lambda}^k_i)$, $k=1,...,r_i$ are all conjugated to the same
element  of type $\begin{pmatrix}
x_i&0 \\
0&x_i^{-1}
\end{pmatrix}$ where $x_i\in{\C}^{\ast}$;

(2) when $\b Q_i$ is non-connected then the fibers of the Seifert
pieces adjacent to $\b\t{Q}$, over $\b Q_i$, are sent to the trivial element under
$\rho$.

\begin{remark} \label{simple}
Notice that even though the representations $\rho$ should be indexed by the components of $p^{-1}({Q})$ we keep the same notations for all of them  for the sake of simplicity.
\end{remark}
When $\b Q_i$ is connected we choose a basis $e_i^1,....e_i^{n_i}$
of the torsion-free submodulus of $H_1(Q_i;{\Z})$ so that
$\lambda_i\in\l e_i^1\r$. Denote by $k_i$ the non-trivial integer
such that $\lambda_i=k_i.e_i^1$ and choose a complex number
$\xi_i=a_ie^{i\theta_i}\in{\C}^{\ast}$ such that $\xi_i^{qk_i}=x_i$
and consider the homomorphism $\eta_i\co\l e_i^1\r\to{\rm
PSL}(2;{\C})$ sending $e_i^1$ to $\begin{pmatrix}
\xi_i&0 \\
0&\xi_i^{-1}
\end{pmatrix}$.
 For each component $\t{Q}_i$  of $p^{-1}({Q}_i)$, consider the representation  given by the composition
 $$\pi_1\t{Q}_i\to H_1(\t{Q}_i;{\Z})\stackrel{p_{\sharp}}{\to} H_1(Q_i;{\Z})\to\l e_i^1\r\stackrel{\eta_i}{\to}{\rm PSL}(2;{\C})$$
where $H_1(Q_i;{\Z})\to\l e_i^1\r$ denotes the natural projection
onto the ${\Z}$-factor spanned by $e_i^1$. Note the composition of
the above homomorphisms sends ${\mu}^k_i$ to the unit of ${\rm
PSL}(2;{\C})$, since ${\mu}^k_i$ was sent to the  torsion part of
$H_1( Q_i, \Z)$ first, and then was sent to the unit under the
natural projection. Therefore the composition of the above
homomorphisms gives rise to a cyclic representation
$$\eta\co\pi_1\hat Q_i\to{\rm PSL}(2;{\C})$$
where $\hat Q_i=\t{Q}_i(\cup_{1}^{r_i}{\mu}^k_i)$, such that
$$\eta(\lambda^k_i)=\begin{pmatrix}
x_i&0 \\
0&x_i^{-1}
\end{pmatrix},$$ and $${\rm vol}(\hat Q_i,\eta)=0$$
Indeed, the former equality can be verified in the following way. The path
$$(\eta_{i})_t(e_i^1)=\begin{pmatrix}
\omega_t e^{it\theta_i}&0 \\
0&\frac{1}{\omega_t}{e^{-it\theta_i}}
\end{pmatrix}$$ of representations $<e_i^1>\to {\rm PSL}(2;{\C})$,
where $\omega_t=ta_i +(1-t)$,  provides   a path of representations
$$\eta_t\co\pi_1\hat Q_i\to{\rm
PSL}(2;{\C})$$ such that $\eta_1=\eta$ and $\eta_0$ is the trivial
representation.

Consider the associated path of flat connections $A_t$. This path
defines a connection $\mathbb{A}$ on the product $\hat
Q_i\times[0,1]$ that is no longer flat but whose curvature
$F^{\mathbb{A}}$ satisfies the equation $F^{\mathbb{A}}\wedge
F^{\mathbb{A}}=0$ (this latter point follows from the fact that
$F^{A_t}=0$ for any $t$). Hence it follows from the construction of
the Chern Simons invariant (paragraph 5.3) combined with the Stokes
formula that $\mathfrak{cs}_{\hat Q_i}^*(A_0)=\mathfrak{cs}^*_{\hat
Q_i}(A_1)=1$. Therefore  ${\rm vol}(\hat Q_i,\eta)=0$. Mind that
Remark \ref{simple} still applies to $\eta$.

Suppose now that $\b Q_j$ isn't connected and denote by $\t{Q}_j$ a component over $Q_j$. Let $\{S_j^l\}_l$ denote the Seifert pieces of  $\t{Q}_j$ adjacent to $\b\t{Q}_j$ with a fixed trivialization $F_j^l\times{\S}^1$. Denote by $\{s_{k,l}\}_{k,l}$ the components of $\b F_j^l\cap\b\t{Q}_j$, by $\{s'_{\kappa,l}\}_{\kappa,l}$ the components of $\b F_j^l\setminus\b F_j^l\cap\b\t{Q}_j$ and by $\t{h_l}$ its ${\S}^1$-fiber.
 Denote by $\hat{F}_j^l$ the surface obtained
from $F_j^l$ after crunching each  boundary component $s'_{\kappa,l}$
into a point and denote by $A$ the set of "singular points" obtained by crunching each components of $\cup_l(\b F_j^l\setminus\b F_j^l\cap\b\t{Q}_j)$ to a point. 
Consider the following relation on $\t{Q}_j$: $x\c{R}y$  iff  either $x$ and $y$ lie on $\o{\t{Q}_j\setminus\cup_l S_j^l}$ or $x$ and $y$ lie on the same ${\S}^1$-fibre of a Seifert piece  $F_j^l\times{\S}^1$ for some $l$. Consider the \emph{crunching} map $\xi\co\t{Q}_j\to\t{Q}_j/\c{R}$. 
  The quotient space $\t{Q}_j/\c{R}$ is homeomorphic to $(\cup_l\hat{F}_j^l)/A$. The fundamental group of $\t{Q}_j/\c{R}$ is an extension of $\ast_l\pi_1\hat{F}_j^l$ by some cycles $\gamma_1,...,\gamma_{\nu}$ obtained after identifying all points of $A$.    Fixing $A/A$ as a base point in $\t{Q}_j/\c{R}$
  we define a representation $\phi\co\pi_1(\t{Q}_j/\c{R})\to{\rm PSL}(2;{\C})$   setting
\begin{eqnarray}
\forall l\ \phi(\t{h_l})=1; &  \forall k,l \ \phi(s_{k,l})=\rho(s_{k,l}); & \forall \kappa,l
\  \phi(s_{\kappa,l}')=1\ {\rm and}\ \forall i\  \phi(\gamma_i)=1.
\end{eqnarray}
Such  a representation can be  defined since  $\hat{F}_j^l$ has a positive genus and
since  each element of ${\rm PSL}(2;{\C})$ is a commutator by
\cite{SW}. Denote by $\hat{Q}_j$ the closed manifold $\t{Q}_j\left(\cup_{l}p^{-1}\left(h^l_j\right)\right)$. By our construction the map $\xi$ induces a homomorphism $\xi_{\ast}\co\pi_1\t{Q}_j\to\pi_1(\t{Q}_j/\c{R})$ that factors through $\pi_1\hat{Q}_j$ giving rise to a representation $\pi_1\hat{Q}_j\to\pi_1(\t{Q}_j/\c{R})\to{\rm PSL}(2;{\C})$, still denoted by $\phi$, and satisfying 
$${\rm vol}\left(\hat
Q_j,\phi\right)=0,$$

 Indeed we first recall the classical long exact sequence of reduced cohomology groups  applied to $\cup_l\hat{F}_j^l\to(\cup_l\hat{F}_j^l)/A$, namely:
$$... \to\t{H}^{n-1}(A)\to\t{H}^n((\cup_l\hat{F}_j^l)/A)\to\t{H}^n(\cup_l\hat{F}_j^l)\to\t{H}^n(A)\to ....$$
This leads to an isomorphism $H^3(\t{Q}_j/\c{R})\simeq H^3(\cup_l\hat{F}_j^l)\simeq\{0\}$. 
  Using the volume defined via the
continuous cohomology (see paragraph 3.2) the induced homomorphism
$$\phi^*\co H^3_{\rm cont}({\rm PSL}(2;{\C}))\to H^3\left(\hat
Q_j\right)$$ factors through $H^3(\t{Q}_j/\c{R})$ that is  is
trivial and
therefore $\phi^*(\omega_{\Hi^3})$ vanishes and

$${\rm vol}_{{\rm
PSL}(2;{\C})}(\hat
Q_j,\phi)=\left|\int_M\phi^{\ast}(\omega_{\Hi^3})\right|=0.$$

To finish the proof we proceed as follows: we choose flat
$\mathfrak{sl}_2({\C})$-connections $A$,  $B_i$ (when $\b Q_i$ is
connected) and $B_j$ (when $\b Q_j$ is non-connected), in normal
form and  corresponding to the representations $\rho|\pi_1\t{Q}$,
$\eta|\pi_1\t{Q}_i$ and $\phi|\pi_1\t{Q}_j$ over each component
$\t{Q}$, $\t{Q}_i$ and $\t{Q}_j$ of $p^{-1}(Q)$, $p^{-1}(Q_i)$ and
$p^{-1}(Q_j)$. By our construction they can be glued together in a
smooth and flat way giving rise to a global representation
$\psi\co\pi_1M\to{\rm PSL}(2;{\C})$. By the additivity principle (a
general for m of (5.17)) we know that $${\rm vol}(M,\psi)=d{\rm
vol}\left(\hat Q,\rho\right)>0$$ where $d$ denotes  the number of
components of $p^{-1}(Q)$. This completes the proof of Proposition
\ref{hyper-virt-nonzero}.
\end{proof}

\section{Manifolds with  positive Gromov simplical volume but vanishing hyperbolic volumes}

\begin{proof}[Proof of Proposition \ref{hyper-zero}] We first
begin by constructing an $1$-edged manifold with a hyperbolic piece
adjacent to a Seifert piece whose hyperbolic volume vanishes.  Let
$M_1$ denote $F\times{\S}^1$ where $F$ is a surface with positive
genus and connected boundary. There is a natural section-fiber basis
$(s,h)\subset \partial M_1$. On the other hand, it follows  from
\cite{HM} that  there are infinitely many one cusped, complete,
finite volume hyperbolic manifolds $M_2$ endowed with a basis
$(\mu,\lambda)\subset \partial M_2$ such that both $M_2(\lambda)$
and $M_2(\mu)$ have zero simplicial volume (because they are
actually connected sums of lens spaces). Denote by $\varphi\co\b
M_1\to \b M_2$ the homeomorphism defined by $\varphi(s)=\mu$ and
$\varphi(h)=\lambda^{-1}$. Let $M_{\varphi}=M_1\cup_{\varphi}M_2$.
Then $M_{\varphi}$ is an one-edged Haken manifold  with positive
simplicial volume. Denote $\c{T}_{M_{\varphi}}$ by $T$.

Let  $\rho\co\pi_1M_{\varphi}\to{\rm PSL}(2;{\C})$ be any
representation and denote by $A$ the resulting connection over
$M_{\varphi}$. Notice that either $\rho(s)$ or $\rho(h)$ is trivial.
Indeed  if $\rho(h)\not=1$, its centralizer $Z(\rho(h))$ must be
abelian in ${\rm PSL}(2;{\C})$. Since $h$ is central in $\pi_1M_1$,
this means that  $\rho(\pi_1M_1)$ is abelian. Since $s$ is
homologically   zero in $M_1$,   then $\rho(s)=1$.

Let $\zeta$ be either $s$ or $h$ so that  $\rho(\zeta)=1$. After
putting $A$ in normal form with respect to $T$, denote by $A_1$ and
$A_2$ the flat connections over $M_1$ and $M_2$ respectively. Since
$\rho(\zeta)$  is trivial then $A_1$ and $A_2$ do extend over
$M_1(\zeta)$ and $M_2(\zeta)$ to flat connections $\hat{A}_1$ and
$\hat{A}_2$,  and thus
$$\mathfrak{cs}^*_{M_{\varphi}}(A)=\mathfrak{cs}^*_{{M}_1(\zeta)}(\hat{A}_1)\times\mathfrak{cs}^*_{{M}_2(\zeta)}(\hat{A}_2)$$
Eventually taking the imaginary part we get
\begin{eqnarray}
{\rm vol}(M_\varphi,\rho)={\rm vol}({M}_1(\zeta),\hat{\rho}_1)+ {\rm
vol}({M}_2(\zeta),\hat{\rho}_2)
\end{eqnarray}
 where $\hat{\rho}_i$
denotes the extension of $\rho|\pi_1M_i$ to $\pi_1M_i(\zeta)$. Since
both ${\rm vol}({M}_1(\zeta),\hat{\rho}_1)$ and  ${\rm
vol}({M}_2(\zeta),\hat{\rho}_2)$ do vanish,  the proof of
Proposition \ref{hyper-virt-nonzero} is complete.
\end{proof}
\begin{remark}
It is worth mentioning that inequality $(7.1)$ still holds replacing
$M_2$ by any one cusped oriented complete hyperbolic manifold. In
this more general case we have ${\rm
vol}({M}_2(\zeta),\hat{\rho}_2)<{\rm vol}M_2$ and therefore, the
volume of the hyperbolic piece of $M_{\varphi}$ is never reached by
the representations of $\pi_1M_{\varphi}$ into ${\rm PSL}(2;{\C})$.
\end{remark}

\section{Volumes of representations of 1-edged 3-manifolds}
In this section we verify Propositions \ref{Seifert-Seifert} and \ref{Seifert-hyperbolic}.
Let $N=Q_-\cup_\tau Q_+$ be a one-edged 3-manifold, where the gluing
map $\tau:
\partial Q_-=T_-\to \partial Q_+=T_+$ is an orientation reversing homeomorphism.
Recall that on each $T_\e$ we fix a basis $T_{\e}(s_\e,h_\e)$ as in paragraph 1.4, (A) and (B). Denote by $A=\begin{pmatrix}
a&b \\
c&d
\end{pmatrix}$ the integral matrix of $\tau$ under the basis $(s_-,h_-)$ and $(s_+,h_+)$,
such that ${\rm det}A=-1$,   and
\begin{eqnarray}
\tau(s_-)=as_++ch_+,\,\,\, \tau(h_-)=bs_++dh_+
\end{eqnarray}
 It follows from
Lemma \ref{cover} and \cite{DW1}
 that $N$ admits a $n$-sheeted covering $\t{N}$, where $n$
depends only on the pieces $Q_-$ and $Q_+$, that induces the (say) $q\times q$-characteristic covering over  $T=\c{T}_N$,  such that:

\begin{enumerate}
{\item $\t{N}=\t Q_-\cup_{\t \tau} \t Q_+$, where $\t Q_\e$ covers $Q_\e$ for $\e=\pm$ and $\t{Q_\e}$  is a product of a surface of genus $\geq p+2$ with
${\S}^1$ if $Q_\e$ is a Seifert piece,}

{\item when $Q_-$ and $Q_+$ are both Seifert manifolds then $\t Q_-$
and $\t Q_+$ can be chosen  connected so that $\t{N}$ is a $p$-edged
manifold with $p\geq 2$} (A $p$-edged 3-manifold is a manifold whose dual graph
consists of 2 vertices and $p$ edges and each edge is shared by the two
vertices),

{\item The basis $T_{\e}(s_\e,h_\e)$ can be lifted, and denoted by
$T_{\e}^j(s_\e^j,h_\e^j)$ for  $j=1,...,p$, and the matrix of the gluing
$\t\tau: T_{-}^j\to T_{+}^j$ is the same than  $A$ for all $j$ under  the
lifted basis, where the $T_{\e}^j$'s denote the components of $\b\t{Q}_{\e}$.}
\end{enumerate}

\subsection{Both pieces are Seifert} This subsection is devoted to the proof of Proposition
\ref{Seifert-Seifert}. It follows from the paragraph above that  it is sufficient to check
the following (simply remember $p=\frac n{q^2}$).

\begin{lemma}\label{computs}
Let $N$ denote a closed  oriented graph manifold satisfying  conditions
(1)-(3). Denote by $G$ the group ${\R}\times_{\Z}\t{{\rm
SL}_2(\R)}$.

(i) if $a=d=0$  then  $8\pi^2p\in{\rm vol}\left(N,G\right)$,

 (ii)  if $ac\not=0$ then $4\pi^2p/|ac|\in{\rm vol}\left(N,G\right)$,

 (iii) if $cd\not=0$ then $4\pi^2p/|cd|\in{\rm vol}\left(N,G\right)$,

 (iv) if  $c=0$  then $4\pi^2p/|b|\in{\rm vol}\left(N,G\right)$.
 \end{lemma}

\begin{proof}
Recall first that in this case $b\not=0$   since $N$ is not a Seifert
manifold.

Denote by $Q_-=F_-\times{\S}^1$  and by $Q_+=F_+\times{\S}^1$ the
two (connected)  Seifert pieces, and recall that $(s_-^i, h^i_-)$ and $(s^i_+,
h^i_+)$ are section-fiber basis of $\partial Q_-$ and $\partial Q_+$
respectively.

Pick some base points $x_-\in{\rm int}Q_-$, $x_+\in{\rm int}Q_+$ and
choose $p$ arcs connecting $x_-$ with $h^i_-\cap s^i_-$, resp. $x_+$
with $h^i_+\cap s^i_+$, to see these elements in $\pi_1Q_-$, resp.
in $\pi_1Q_+$.

(i) Suppose first $a=d=0$. Then $b=c=\pm 1$. We get directly a
representation $\rho\co\pi_1N\to\t{{\rm SL}_2(\R)}$ defined by
$\rho(s^i_{\e})={\rm sh}(1)$ and $\rho(h^i_{\e})={\rm sh}(1)$. Such
a representation does exist since the genus of $F_+$ and $F_-$ $\geq
p+2$. The additivity principle gives
$${\rm vol}(N,\rho)={\rm vol}(Q_+(-1,1),\rho|)+{\rm vol}(Q_-(1,-1),\rho|)$$
By Proposition \ref{SvSm} we know that ${\rm vol}(Q_{\e}(-1,1),\rho|)={\rm vol}(Q_{\e}(1,-1),\rho|)=4\pi^2p$ hence the proof of point (i) follows.

(ii) Let's now assume $ac\not=0$.  Then the closed manifold
$Q_+((a,c),....,(a,c))$ is still a Seifert manifold with Euler
number $\pm pc/a$.

In Proposition \ref{SvSm} by choosing
$z_1=...=z_p=n=0$ and $n_1=...=n_p=1$  we have a representation
$\rho_+ \co\pi_1Q_+\to{\R}\times_{\Z}\t{{\rm SL}_2(\R)}$ such that
\begin{eqnarray}
\rho_+(s^i_+)=\o{\left(0, {\rm sh}\left(-\frac{1}{a}\right)\right)},
\rho_+(h^i_+)=\o{\left(\frac{1}{c},1\right)}
\end{eqnarray} since
$g(F_+)
>p$, the condition (4.1) is clearly satisfied.
According to (4.2) and (4.3), we have
\begin{eqnarray}
 \rho_+(s^i_-)=\o{(0, {\rm sh}(0))}, \,\,\,\rho_+(h^i_-)=\o{\left(\frac{d}{c},{\rm
sh}\left(-\frac{b}{a}\right)\right)}
\end{eqnarray}
Now we can extend $\rho_+$  to
 ${\rho}\co\pi_1N \to{\R}\times_{\Z}\t{{\rm SL}_2(\R)}$ merely by sending the whole subgroup  $\pi_1F_-$ of  $\pi_1Q_-$ to the unit of $G$.
To apply the additive principle, we need further to construct a
$c$-fold cyclic covering  $q\co \t{N}_c\to N$ so that  the induced
representation $\t{\rho}={\rho}\circ q_*\co\pi_1\t{N}_c\to
{\R}\times_{\Z}\t{{\rm SL}_2(\R)}$ has image in $\t{{\rm
SL}_2(\R)}$.  This covering $p\co \t{N}_c\to N$ can be obtained  by
combining  the covering $p_+: \t Q_+ \to Q_+ $ defined by
$\varphi_+\co\pi_1Q_+\to\z{c}$ with $\varphi_+(s^i_+)=\o{0},
\varphi_+(h^i_+)=\o{1}$; and  the covering $p_-: \t Q_- \to Q_-$
defined by $\varphi_-\co\pi_1Q_-\to\z{c}$ with
$\varphi_-(s^i_-)=\o{0}, \varphi_-(h^i_-)=\o{d}$.
 It is easy to verify that

$$\t{\rho}(\t{s}^i_+)=\left(0, {\rm sh}\left(-\frac{1}{a}\right)\right), \t{\rho}\left(\t{h}^i_+\right)=(0, {\rm sh}(1))$$
$$\t{\rho}(\t{s}^i_-)={(0, {\rm sh}(0))}, \t{\rho}(\t{h}^i_-)=\left(0, {\rm sh}\left(-\frac{1}{a}\right)\right)$$
Where the $(\t{s}^i_{\e},\t{h}^i_{\e})$'s are the lifts of the $({s}^i_{\e},{h}^i_{\e})$'s into $\t{N}_c$. Hence indeed  $\t{\rho}$  takes its values in $\t{{\rm
SL}_2(\R)}$. Since $s^i_-=as^i_++ch^i_+$ then we get
$\t{s}^i_-=a\t{s}^i_++\t{h}^i_+$. So the Euler number $e(\t{Q}_+((a,1),...,(a,1))=
p/a$. Now according to (4.4) and $|e|=|p/a|$, we have
\begin{eqnarray}
{\rm vol}(\t{Q}_+((a,1),...,(a,1),\t{\rho})=4\pi^2\frac{1}{|e|}e^2=4\pi^2
|e|=4\pi^2|p/a|
\end{eqnarray}
On the other hand it is clear that ${\rm vol}\left(\t{Q}_-\left(\left(1,0\right),...,\left(1,0\right)\right),\t{\rho}\right)=0$.

Since $\t{{\rm SL}_2(\R)}$ is contractible, as we did in \cite{DW2}, we can
apply the additive principle to compute ${\rm
vol}(\t{N}_c,\t{\rho})$. Precisely by (5.17) we have
$${\rm vol}(\t{N}_c,\t{\rho})={\rm vol}(\t{Q}_+((a,1),...,(a,1)),\t{\rho})+{\rm vol}(\t{Q}_+((1,0),...,(1,0)),\t{\rho})=4\pi^2|p/a|$$
Since ${\rm vol}(\t{N}_c,\t{\rho})=|c|{\rm vol}(N,{\rho})$, we
have ${\rm vol}({N},{\rho})=4\pi^2p/|ac|$.
 This proves point (ii).

(iii) The proof  is the same as  that of (ii)  just  by replacing
$A$ by $A^{-1}$.

(iv) We will get directly a representation $\rho\co\pi_1N\to\t{{\rm
SL}_2(\R)}$ by first setting $\rho(h^i_-)={\rm sh}(1)$,
$\rho(s^i_-)={\rm sh}(\e_a/b)$ for $i=1,...,p$, where $\e_a$ denotes
the sign of $a$. Since $|p/b|\leq p<2g-2$  such a representation
exists. On the other hand, we get a representation
$\psi\co\pi_1Q_+\to\t{{\rm SL}_2(\R)}$ by setting $\psi(h^i_+)=1$
and $\psi(s^i_+)={\rm sh}(1/b)$. Again such a representation does
exist. Then we can use the additive principle to get ${\rm
vol}\rho=4\pi^2|p/b|$. This completes the proof of Lemma
\ref{computs} and therefore the proof of Proposition
\ref{Seifert-Seifert}.
\end{proof}

\subsection{Both Seifert and hyperbolic pieces do appear}
This subsection is devoted to the proof of Proposition
\ref{Seifert-hyperbolic}. Let's say that $Q_+$ is hyperbolic and
$Q_-$ is Seifert. The proof below is divided in two cases, which
involve quite different arguments in certain stages.

As in \cite{KK}   denote  by $D$ the  space of deformations
of hyperbolic structure on ${\rm int}Q_+$ near the complete one $d_0\in D$. Since $Q_+$ has only one cusp  there are functions
$$(\alpha, \beta)\co D^{\ast}=D\setminus\{d_0\}\to{\C}^{2}$$
such that  for each $d\in D^{\ast}$ there exists a representation
$\rho^+_d\co\pi_1Q_+\to{\rm PSL}(2;{\C})$ induced on the boundary by
the representation
 \begin{eqnarray}
 s_+\mapsto\begin{pmatrix}
e^{2i\pi\alpha}&0 \\
0&e^{-2i\pi\alpha}
\end{pmatrix}\ {\rm and}\ h_+\mapsto\begin{pmatrix}
e^{2i\pi\beta}&0 \\
0&e^{-2i\pi\beta}
\end{pmatrix}
\end{eqnarray}
In this situation the map $D^{\ast}\ni
d\mapsto(\alpha,\beta)\in{\C}^2$ is a lifting of the composition map
$$ D^*\to\c{ER}^0_M({\rm PSL}(2;{\C}))\to\c{ER}^0_{\b M}({\rm
PSL}(2;{\C}))$$ By the Thurston Hyperbolic Dehn filling Theorem
there is a constant $C>0$ such that if (say) $\|(a,c)\|_2>C$ then
there exists $d\in D^{\ast}$ such that
 \begin{eqnarray}
a\alpha+c\beta=1/2
\end{eqnarray}
Let $V={\bf D}^2\times{\S}^1$ be a solid torus endowed with the standard meridian-parallel basis $(m,l)$. The representation $\rho^+_d$ extends to a complete
and faithful representation
$$\hat{\rho}^+_d\co\pi_1Q_+(a,c)\to{\rm PSL}(2;{\C})$$
where $Q_+(a,c)$ is obtained by gluing $\b V$ to $\b Q_+$ identifying
the meridian of $V$ with the curve $as_++ch_+$. Let $\hat{A}^+_d$
denote the connection over ${Q}_+(a,c)$ in normal hyperbolic form
over $\b{Q}_+$ which decomposes into $A^+_d\cup A^0_d$ over $Q_+\cup
V$. We choose a lifting $L(\rho^+_d)$  such that $$A^+_d|\b Q_+=(i\alpha
dx+i\beta dy)\otimes X$$ in the basis $(s_+,h_+)$. By (8.5) and (8.6) this
means that $$A^0_d|\b V=\left(i\frac{1}{2} dx+i(b\alpha+d\beta)
dy\right)\otimes X$$ in the basis $(m,l)$. The similar construction
can be done replacing $(a,c)$ by $(b,d)$.

\vskip 0.3 true cm{\bf 1. Pinching the section $s_-=\partial F_-$.}
Denote by $\rho_+\co\pi_1\t{Q}_+\to{\rm PSL}(2;{\C})$ the
representation defined by the composition  $\rho^+_d\circ p_{\ast}$,
where $p\co\t{Q}_+\to Q_+$ is the $q\times q$-characteristic
covering map defined above. This representation induces the
following relations by (8.1) and (8.5) : $\rho_+(s^j_-)$ is the
trivial element   and $\rho_+(h^j_-)$ is sent
to
$$\begin{pmatrix}
e^{2i\pi q(b\alpha+d\beta)}&0 \\
0&e^{-2i\pi q(b\alpha+d\beta)}
\end{pmatrix}$$ in ${\rm PSL}(2;{\C})$. Thus there exists a global representation $\rho\co\pi_1\t{N}\to{\rm PSL}(2;{\C})$ such that $\rho|\pi_1\t{Q}_+=\rho_+$.  Denote $\rho_-=\rho|\pi_1\t{Q}_-$.  Let $A$ be a flat connection in hyperbolic normal form with respect to $\c{T}_{\t{N}}$ such that $A=A_-\cup A_+$ where $A_-$, resp. $A_+$, is the restriction of $A$ over $\t{Q}_-$, resp. $\t{Q}_+$.  Notice that $A_+=(p|\t{Q}_+)^{\ast}(A^+_d)$.

For each $j=1,...,p$ we identify the meridian of a solid torus
$V^j_{\pm}={\bf D}^2\times{\S}^1$ with $s^j_-$ and with
$as^j_++ch^j_+$ and then we get  closed manifolds $\t{Q}_-(1,0)$ and
$\t{Q}_+(a,c)$ where $A_-$ and $A_+$, extend to  flat connections
$\hat{A}_-$ and $\hat{A}_+$ such that
$$\mathfrak{cs}^*_{\t{N}}(A)=\mathfrak{cs}^*_{\t{Q}_+(a,c)}(\hat{A}_+)\times\mathfrak{cs}^*_{\t{Q}_-(1,0)}(\hat{A}_-)$$
by the additivity principle.
\begin{remark}\label{fondamental}
Again $\hat{A}_+|V^j_+=(p|{V}^j_+)^{\ast}(A^0_d)$ but mind that $p|{V}^j_+\co{V}^j_+\to V$ is a $q\times q$-characteristic covering branched along the core of the solid torus. Therefore the lifting of the representations induced on the $V_j^{+}$'s is $(1/2,q(b\alpha+d\beta))$.
\end{remark}

Denote by $\hat{\rho}_+$, $\hat{\rho}_-$ the extension of $\rho_+$ and $\rho_-$ to $\pi_1\t{Q}_+(a,c)$ and $\pi_1\t{Q}_-(1,0)$. Splitting the former equality into real and imaginary parts according to (5.17) we get then
$${\rm vol}(\t{N},\rho)={\rm vol}(\t{Q}_+(a,c),\hat{\rho}_+)+{\rm vol}(\t{Q}_-(1,0),\hat{\rho}_-)$$
Since $\t{Q}_-$ is by construction a product of a surface with the circle then $${\rm vol}(\t{Q}_-(1,0),\hat{\rho}_-)=0$$
On the other hand, notice that, a priori, the representation ${\rho}_+$ lies by no means in the deformation space of hyperbolic structures of $\t{Q}_+$. Nevertheless we need a geometric interpretation of ${\rm vol}(\t{N},\rho)={\rm vol}(\t{Q}_+(a,c),\hat{\rho}_+)$. It follows from Remark \ref{fondamental} and from Corollary \ref{adding} applied $p$ times that
$$\mathfrak{cs}^*_{\t{Q}_+(a,c)}(\hat{A}_+)=\mathfrak{cs}^*_{\t{Q}_+}(A_+)\exp(-4i\pi pq(b\alpha+d\beta))$$
Since $\rho_+$ is induced by restriction from  $\rho^+_d$ then
$\mathfrak{cs}^*_{\t{Q}_+}(A_+)= (\mathfrak{cs}^*_{{Q}_+}(A^+_d))^{pq^2}$ therefore
$$\mathfrak{cs}^*_{\t{Q}_+(a,c)}(\hat{A}_+)=(\mathfrak{cs}^*_{{Q}_+}(A^+_d))^{pq^2}\exp(-4i\pi pq(b\alpha+d\beta))$$
Again by Corollary \ref{adding}
$$\mathfrak{cs}^*_{Q_+(a,c)}(\hat{A}_d^+)=\mathfrak{cs}^*_{{Q}_+}(A^+_d)\exp(-4i\pi(b\alpha+d\beta))$$ This leads to
$$\mathfrak{cs}^*_{\t{Q}_+(a,c)}(\hat{A}_+)=(\mathfrak{cs}^*_{Q_+(a,c)}(\hat{A}^+_d))^{pq^2}\exp(4i\pi pq(q-1)(b\alpha+d\beta))$$
and accordingly by equality (5.15) and Remark \ref{volhh} that can be applied to $\mathfrak{cs}_{Q_+(a,c)}(\hat{A}^+_d)$ for $\hat{\rho}^+_d$ is a faithful and discrete representation
\begin{eqnarray*}
\mathfrak{cs}^*_{\t{Q}_+(a,c)}(\hat{A}_+)& =\exp(2i\pi\mathfrak{cs}_{\rm L.C.}{Q}_+(a,c))^{pq^2}\times\\
&\exp\left(\frac{2pq^2}{\pi}{\rm vol}Q_+(a,c)+4i\pi pq(q-1)(b\alpha+d\beta)\right)
 \end{eqnarray*}

Again using (5.17) and splitting the former equality into real and imaginary parts we get eventually
$${\rm vol}(\t{N},\rho)={\rm vol}(\t{Q}_+(a,c),\hat{\rho}_+)=pq^2{\rm vol}Q_+(a,c)+\frac{\pi pq(q-1)}{2}{\rm length}(\gamma)$$
where $\gamma$ is the geodesic added to $Q_+$ to complete the cusp with respect to the $(a,c)$-Dehn filling.

\vskip 0.3 true cm {\bf 2. Pinching the fiber $h_-$.} By the
Thurston Hyperbolic Dehn filling Theorem   there is a constant $C>0$
such that if  $\|(b,d)\|_2>C$ then there exists $d\in D^{\ast}$ such
that (i)
\begin{eqnarray}
b\alpha+d\beta=1/2
\end{eqnarray}
 Let $V={\bf D}^2\times{\S}^1$ be a solid torus endowed with the standard meridian-parallel basis $(m,l)$. The representation $\rho_d$ extends to a complete and faithful representation $\hat{\rho}_d\co\pi_1Q_+(b,d)\to{\rm PSL}(2;{\C})$, where $Q_+(b,d)$ is obtained by gluing $\b V$ to $\b Q_+$ identifying the meridian of $V$ with the curve $bs_++dh_+$. Let $\hat{A}^+_d$ denote the connection over ${Q}_+(b,d)$ in normal hyperbolic form over $\b{Q}_+$ which decomposes into $A^+_d\cup A^0_d$ over $Q_+\cup V$. We choose a lifting $L(\rho^+_d)$ such that $$A^+_d|\b Q_+=(i\alpha dx+i\beta dy)\otimes X$$ in the basis $(s_+,h_+)$. By equation (8.7) this means that $$A^0_d|\b V=\left(i\frac{1}{2} dx+i(a\alpha+c\beta) dy\right)\otimes X$$ in the $(m,l)$ basis.

Denote by $\rho_+\co\pi_1\t{Q}_+\to{\rm PSL}(2;{\C})$ the representation defined by the composition  $\rho^+_d\circ p_{\ast}$, where  $p\co\t{Q}_+\to Q_+$ is the $q\times q$-characteristic covering map defined above. This representation induces the following relations: $\rho_+(h^j_-)$ is the trivial element and $\rho_+(s^j_-)$ is sent to    $$\begin{pmatrix}
e^{2i\pi q(a\alpha+c\beta)}&0 \\
0&e^{-2i\pi q(a\alpha+c\beta)}
\end{pmatrix}$$ in ${\rm PSL}(2;{\C})$. Again since by   \cite{SW} each element of ${\rm PSL}(2;{\C})$ is a commutator,  there exists a global representation $\rho\co\pi_1\t{N}\to{\rm PSL}(2;{\C})$ such that $\rho|\pi_1\t{Q}_+=\rho_+$.  Denote $\rho_-=\rho|\pi_1\t{Q}_-$.  Let $A$ be a flat connection in hyperbolic normal form with respect to $\c{T}_{\t{N}}$ such that $A=A_-\cup A_+$ where $A_-$, resp. $A_+$, is the restriction of $A$ over $\t{Q}_-$, resp. $\t{Q}_+$.
For each $j=1,...,p$ we identify the meridian of a solid torus ${\bf D}^2\times{\S}^1$ with $h^j_-$ and with $bs^j_++dh^j_+$  then we get  closed manifolds $\t{Q}_-(0,1)$ and $\t{Q}_+(b,d)$ where $A_+$ and $A_-$ extend to  flat connections  $\hat{A}_+$ and $\hat{A}_-$ such that
$$\mathfrak{cs}^*_{\t{N}}(A)=\mathfrak{cs}^*_{\t{Q}_+(b,d)}(\hat{A}_+)\times\mathfrak{cs}^*_{\t{Q}_-(0,1)}(\hat{A}_-)$$
Keeping the same notation as in the previous section, by splitting the former equality into real and imaginary parts according to (5.17), we get
 $${\rm vol}(\t{N},\rho)={\rm vol}(\t{Q}_+(b,d),\hat{\rho}_+)+{\rm vol}(\t{Q}_-(0,1),\hat{\rho}_-)$$
Since $\t{Q}_-$ is by construction  a connected sum of ${\S}^2\times{\S}^1$-factors, then  ${\rm vol}(\t{Q}_-(0,1),\hat{\rho}_-)=0$.
Applying the same arguments as in the former section we get
$${\rm vol}(\t{N},\rho)=pq^2{\rm vol}Q_+(b,d)+\frac{\pi pq(q-1)}{2}{\rm length}(\gamma)$$
where $\gamma$ is the geodesic added to $Q_+$ to complete the cusp with respect to the $(b,d)$-Dehn filling.
This proves Proposition \ref{Seifert-hyperbolic}.

\end{document}